\documentclass{article}

\usepackage{amsrefs,epsfig,amsmath,amssymb,amsthm,latexsym,subfig,booktabs}
\usepackage[normalem]{ulem}

\usepackage{chngcntr}
\usepackage{apptools}
\AtAppendix{\counterwithin{lemma}{section}}
\AtAppendix{\counterwithin{theorem}{section}}

\def\rmd{\mathrm{d}}
\def\rmT{\mathrm{T}}
\def\rme{\mathrm{e}}
\def\sinc{\mathrm{sinc}}
\def\supp{\text{supp}}
\def\GGG{\mathcal{G}}
\def\NN{\mathbb{N}}
\def\RR{\mathbb{R}}
\def\ZZ{\mathbb{Z}}
\def\wphi{\widehat{\phi}}
\def\ps#1#2#3{\raisebox{-0.75em}{\!\!\!\includegraphics[scale=0.375]{pseudo-spline#1#2#3}\!\!\!}}
\def\bfdelta{{\boldsymbol{\delta}}}
\newcommand{\bfa}{{\boldsymbol{a}}}
\newcommand{\bfb}{{\boldsymbol{b}}}
\newcommand{\bfA}{{\boldsymbol{A}}}
\newcommand{\bfM}{{\boldsymbol{M}}}
\newcommand{\bfc}{{\boldsymbol{c}}}
\newcommand{\bfd}{{\boldsymbol{d}}}
\newcommand{\bff}{{\boldsymbol{f}}}

\newtheorem{theorem}{Theorem}
\newtheorem{lemma}{Lemma}
\newtheorem{corollary}{Corollary}

\theoremstyle{definition}
\newtheorem{definition}{Definition}

\theoremstyle{remark}
\newtheorem{remark}{Remark}
\newtheorem{example}{Example}

\begin{document}

\title{Symbols and exact regularity\\ of symmetric pseudo-splines of any arity}

\author{Georg Muntingh\footnote{SINTEF ICT, PO Box 124 Blindern, 0314 Oslo, Norway}}

\maketitle
                    
\begin{abstract}
Pseudo-splines form a family of subdivision schemes that provide a natural blend between interpolating schemes and approximating schemes, including the Dubuc-Deslauriers schemes and B-spline schemes. Using a generating function approach, we derive expressions for the symbols of the symmetric $m$-ary pseudo-spline subdivision schemes. We show that their masks have positive Fourier transform, making it possible to compute the exact H\"older regularity algebraically as a logarithm of the spectral radius of a matrix. We apply this method to compute the regularity explicitly in some special cases, including the symmetric binary, ternary, and quarternary pseudo-spline schemes.
\end{abstract}



\section{Introduction}\label{sec:intro}
Subdivision is a recursive method for generating curves, surfaces and other geometric objects. Rather than having a complete description of the object of interest at hand, subdivision generates the object by repeatedly refining its description starting from a coarse set of control points. See the seminal work~\cite{Cavaretta.Dahmen.Micchelli91} and comprehensive survey \cite{Dyn.Levin02}.

Pseudo-splines form a family of subdivision schemes that provide a natural blend between interpolating schemes and approximating schemes, including the Dubuc-Deslauriers schemes and B-spline schemes. Primal binary pseudo-splines were introduced in the context of framelets \cite{Daubechies.Han.Ron.Shen03, Dong.Shen07}, followed by the introduction of a family of schemes analogous to the Dubuc-Deslauriers schemes \cite{Dyn.Floater.Hormann04} and its generalization to a family of dual binary pseudo-splines \cite{Dyn.Hormann.Sabin.Shen08}. After this, pseudo-splines were introduced in the context of surface schemes \cite{Deng.Hormann14}, nonstationary schemes \cite{Conti.Gemignani.Romani16}, and $m$-ary schemes \cite{Conti.Hormann11}, where the $m$-ary pseudo-splines were defined (see Definition \ref{def:pseudosplines}) in terms of their ability to generate and reproduce polynomials.

Polynomial reproduction is important, because it is tied to the approximation order of the scheme. If a subdivision scheme reproduces polynomials up to degree $l$, then its application to initial data sampled from a function of class $C^l$ will reproduce the Taylor polynomial of degree $l$ and thus yield approximation order $l+1$. In this sense, pseudo-splines balance approximation power, regularity of the limit function, and support size. By resulting to schemes of higher arity or giving up symmetry of the limit function, better trade-offs between these criteria can be obtained (cf. \cite[Table 1]{Mustafa.Khan09}, \cite[\S 5.5--5.6]{Conti.Hormann11}).

While fast recursive algorithms are known for computing the symbol of a pseudo-spline, an explicit formula has so far only appeared in the binary case \cite{Conti.Hormann11}. In this paper we derive an explicit formula of the pseudo-spline symbol for any arity, in terms of a generating function involving Chebyshev polynomials of the second kind. The generating function framework allows us to juggle a three-parameter family of pseudo-splines, involving the arity $m$, degree of polynomial generation $n$, and degree of polynomial reproduction $l$.

It is well known \cite{Rioul92, Han02} that interpolatory schemes with positive Fourier transform admit exact formulas for the (H\"older) regularity, and it was recently shown that the same technique can also be applied to non-interpolatory schemes \cite{Floater.Muntingh13}. In the self-contained appendix to this paper, we show that these results generalize to schemes of arbitrary arity, a result that can ultimately be attributed to the validity of the finiteness conjecture for the joint spectral radius of subdivision submatrices derived from schemes with positive Fourier transform \cite{Charina14, Moeller15, Moeller.Reif14}. Together with the explicit formulas for the symbol established in this paper, this allows us to compute the corresponding regularities algebraically as logarithms of roots of univariate polynomials, which are provided numerically in several tables.

The method in this paper has been implemented in \texttt{Sage} and tested for the examples in this paper (and many other examples). Moreover, in many cases the results proven in this paper were verified symbolically. The resulting complementary worksheet, which includes an interactive ``pseudo-spline explorer applet'', can be tried out online in \texttt{SageMathCloud} following a link on the website of the author \cite{WebsiteGeorg}. 

The plan of the paper is as follows. In Section \ref{sec:background} we first recall some basic notions and facts from subdivision. Then, in Section \ref{sec:pseudo-splines}, we define the $m$-ary pseudo-splines and derive their symbols as a truncated generating function. Finally in Section \ref{sec:regularity} we proceed to analyze their regularity, applying the method developed in the Appendix,  followed by a conclusion in Section~\ref{sec:conclusion}.

\section{Background}\label{sec:background}
\subsection{Basic notation}
For ease of reference, we first describe the notation used throughout the paper:
\begin{itemize}
\item $\NN_0\subset\NN\subset \ZZ$ the nonnegative, positive, and entire set of integers;
\item $\ell$ the subdivision level of the data;
\item $i$ the imaginary unit;
\item $\bfa = [a_j]_{j\in \ZZ}, \bfb = [b_j]_{j\in \ZZ},\ldots$ sequences;
\item $a(z), b(z), \ldots$ Laurent polynomials, $z$-transforms of the sequences $\bfa, \bfb,\ldots$;
\item $A(\xi), B(\xi),\ldots$ discrete-time Fourier transforms of the sequences $\bfa, \bfb, \ldots$;
\item $m = 2m' + \epsilon$ the arity of the scheme, with $\epsilon\in\{0,1\}$ the parity of the arity;
\item $n$ and $l = 2l' + 1$ parameters of the pseudo-spline scheme.
\end{itemize}
	
\subsection{Subdivision schemes}
A linear, univariate, stationary, uniform \emph{subdivision scheme} $S_\bfa$, of arity \mbox{$m\geq 2$} and with \emph{mask} $\bfa = [a_i]_{i\in \ZZ}$ a compactly supported real sequence, is based on repeatedly applying the \emph{refinement rule}
\begin{equation}\label{eq:scheme}
 f_{\ell+1,j} = \sum_k a_{j-mk} f_{\ell, k},
\end{equation}
starting from \emph{initial data} $\bff_0 = [f_{0,j}]_{j\in \ZZ}$. In terms of polynomial arithmetic,
\[ a(z) := \sum_{j\in \ZZ} a_j z^j,\qquad f_\ell(z) := \sum_j f_{\ell,j} z^j,\qquad f_{\ell+1}(z) = a(z) f_\ell(z^m), \]
where the mask $\bfa$ and $\bff_\ell$, the \emph{data at level $\ell$}, have been encoded by the \emph{symbols} $a(z)$ and $f_\ell(z)$ as formal Laurent polynomials. A third representation is
\[ A(\xi) := a(\rme^{-i\xi}),\qquad F_\ell(\xi) := f_\ell(\rme^{-i\xi}), \qquad F_{\ell+1}(\xi) = A(\xi) F_\ell(m\xi), \]
as discrete-time Fourier transforms of the mask $\bfa$ and data $\bff_\ell$ at level $\ell$.

\subsection{Odd and even symmetry}
If $j_0$ is the minimal integer and $j_1$ the maximal integer for which $a_{j_0}, a_{j_1}\neq 0$, the mask has \emph{length} $j_1 - j_0 + 1$ and is \emph{supported} on $[j_0, j_1]$.

We distinguish two types of symmetry, depending on whether the mask has odd or even length. A subdivision scheme $S_\bfa$ is \emph{odd symmetric} if there exists an index $j_0$ such that $a_{j_0 + j} = a_{j_0 - j}$ for all $j$, or equivalently, if $a(z) = z^{2j_0} a(z^{-1})$. Similarly $S_\bfa$ is \emph{even symmetric} if there exists an index $j_0$ such that $a_{j_0 + j} = a_{j_0 + 1 - j}$ for all $j$, or equivalently, if $a(z) = z^{2j_0 - 1} a(z^{-1})$. In either case, the scheme and symbol are called \emph{symmetric}.

A mask of length $2L+1-\varepsilon$, $\varepsilon\in \{0,1\}$, is \emph{centered} if it takes the form $\bfa = [a_{-L}, \ldots, a_0, \ldots, a_{L-\varepsilon}]$, with $a_{-L}, a_{L-\varepsilon}\neq 0$. (Note that some authors center their masks of even length at $1/2$ instead of $-1/2$.) 

\begin{remark}\label{rem:symmetryfactorisation}
Suppose a symbol admits the factorization $a(z) = c(z) b(z)$, with both $a(z)$ and $c(z)$ symmetric. Then, for some integers $j_0, j_1$, one finds that
\[ b(z) = \frac{a(z)}{c(z)} = \frac{z^{j_0} a(z^{-1})}{z^{j_1} c(z^{-1})} = z^{j_0 - j_1} b(z^{-1}) \]
is symmetric as well. For instance, whenever a symbol $a(z)$ is even symmetric, it has a root at $z = -1$ and therefore admits the factorization $a(z) = (1+z) b(z)$, with $b(z)$ odd symmetric. 
\end{remark}

\subsection{Parametrization}
At each level $\ell$, we consider the data $\bff_\ell = [f_{\ell,j}]_j$ as the values of a continuous interpolant $L_\ell$ at parameters $t_{\ell,j}$, with $t_{\ell, j+1} - t_{\ell,j} = m^{-\ell}$. We choose $L_\ell$ to be piecewise linear, but remark that in some cases piecewise polynomials of higher degree can be used to simplify the analysis of the scheme \cite{Floater11, Floater.Siwek13}. 

It has been shown \cite{Conti.Hormann11} that for reproduction (definition below) of linear polynomials it is necessary to introduce a constant \emph{(parameter) shift}
\begin{equation} \label{eq:tau}
\tau := a'(1)/m,
\end{equation}
between the levels, in the sense that the parameters $t_{\ell,j}$ are determined by
\begin{equation}\label{eq:correctparametrization}
t_{\ell, j} = t_{\ell, 0} + \frac{j}{m^\ell},\qquad t_{\ell+1,0} = t_{\ell,0} - \frac{\tau}{m^{\ell+1}}, \qquad
j\in \ZZ,\qquad \ell \in \NN_0.
\end{equation}

For schemes with centered masks, we consider two types of parametrizations (see \cite[Definition 5.5]{Conti.Hormann11}):
\begin{itemize}
\item \emph{primal} (or \emph{standard}) \emph{parametrization}: $\tau = 0$ and $t_{0,0} = 0$. This corresponds to attaching the initial data $f_{0,j}$ to the integers $t_{0,j} = j$.
\item \emph{dual parametrization}: $\tau = -1/2$ and $t_{0,0} = \tau/(m-1)$. This corresponds to attaching the data $f_{\ell,j m^\ell}$ to the integers $j$ in the limit $\ell\to\infty$.
\end{itemize}
The reason is that for odd (respectively even) symmetric schemes, only the primal (respectively dual) parametrization can reproduce linear polynomials \cite[Corollary 5.7]{Conti.Hormann11}.

\subsection{Limit function}
The scheme $S_\bfa$ is \emph{convergent}, if, for any choice of initial data $\bff_0$, the sequence $[L_\ell]_\ell$ converges in the uniform norm to some function $f = f_\bfa$, called a \emph{limit function} of the scheme. The scheme is \emph{interpolatory} if $a_{mj} = \delta_{j,0}$, in which case $f$ interpolates $\bff_0$. The schemes in this paper are assumed to be \emph{nonsingular}, meaning $f = 0$ precisely when $\bff_0 = 0$. By linearity of the refinement rule \eqref{eq:scheme}, it suffices to study the \emph{cardinal limit function} $\phi = \phi_\bfa$ obtained by taking as initial data $\bff_0 = \bfdelta = [\delta_{0,j}]_j$, with $\delta_{k,j}$ the Kronecker delta. 

\subsection{Size of the support}
The support of the cardinal limit function $\phi_\bfa$ of an $m$-ary subdivision scheme is determined by the support of the mask $\bfa = [a_j]_j$ and the arity $m$. In particular, if $\bfa = [a_0, \ldots, a_N]$, then $\phi_\bfa$ has support (cf. \cite{Conti.Hormann11, Ivrissimtzis.Sabin.Dodgson04})
\[ \supp(\phi_\bfa) = \left[0, \frac{N}{m-1} \right]. \]

\subsection{Convergence}
A necessary condition for convergence \cite{Han.Jia98} of the scheme $S_\bfa$ is
\begin{subequations}\label{eq:convergence}
\begin{equation}\label{eq:convergence1}
\sum_k a_{mk} = \sum_k a_{mk + 1} = \cdots = \sum_k a_{mk + m - 1} = 1,
\end{equation}
and we will make this assumption. This condition can be expressed in terms of the symbol as
\begin{equation}\label{eq:convergence2}
a(1) = m,\qquad a(\zeta^1_m) = a(\zeta^2_m) = \cdots = a(\zeta^{m-1}_m) = 0,
\end{equation}
\end{subequations}
where $\zeta_m := \exp(2\pi i/m)$. Under this condition, it follows from the refinement rule that constant polynomials are reproduced. 

\subsection{H\"older regularity}
The limit function $f$ has \emph{(H\"older) regularity} $\alpha$, $0 < \alpha < 1$, written $f\in C^\alpha$, if there exists a constant $K$ such that
\[ |f(x) - f(y)| \le K |x - y|^\alpha,\qquad \text{for all } x,y \in \RR.\]
Moreover, we write $f\in C^{q+\alpha}$ for $q\in \NN_0$ and $0 < \alpha < 1$, if $f$ is $q$ times continuously differentiable, and $f^{(q)} \in C^\alpha$. Correspondingly, we say that the scheme \eqref{eq:scheme} has H\"older regularity $\gamma$ for some real $\gamma \geq 0$, if 
\begin{itemize}
\item for every $\beta < \gamma$, $f\in    C^\beta$ for all initial data $\bff_0$, and
\item for every $\beta > \gamma$, $f\notin C^\beta$ for some initial data $\bff_0$.
\end{itemize}

\subsection{Polynomial generation}
The scheme $S_\bfa$ is said to \emph{generate} polynomials up to degree $n$ if any polynomial of degree at most $n$ is the limit function for some choice of the initial data $\bff_0$. This happens precisely when
\begin{subequations}\label{eq:polynomialgeneration}
\begin{equation}\label{eq:polynomialgeneration1}
a(1)=m,\qquad a^{(k)}(\zeta^j_m) = 0,\qquad j = 1,\ldots, m-1,\qquad k = 0,\ldots, n,
\end{equation}
or, equivalently, when the symbol $a(z)$ admits a Laurent polynomial factorization
\begin{equation}\label{eq:polynomialgeneration2}
a(z) = m\sigma_m(z)^{n+1} b(z),\qquad b(1) = 1,
\end{equation}
\end{subequations}
where $\sigma_m$ is the \emph{$m$-ary smoothing factor} defined by
\begin{equation}\label{eq:sigma}
\sigma_m(z) := \frac{1 + z + \cdots + z^{m-1}}{m} = \frac{1 - z^m}{m(1 - z)}.
\end{equation}
We will later use that
\begin{equation}\label{eq:sigmaderivative}
\sigma_m^{(k)}(1) = \frac{(m-1)\cdots (m-k)}{k+1}, \qquad k\geq 0.
\end{equation}
We refer to $b(z)$ as the \emph{derived symbol}. The special case $b(z) = 1$ yields the \emph{$m$-ary degree $n$ B-spline scheme}.

\subsection{Polynomial reproduction}
The scheme $S_\bfa$ is said to \emph{reproduce} polynomials up to degree $l$ if any polynomial of degree at most $l$ is the limit function for the initial data $\bff_0$ sampling this polynomial. This happens \cite[Theorem 4.3]{Conti.Hormann11} precisely when \eqref{eq:polynomialgeneration} holds and, in addition,
\begin{subequations}\label{eq:ReproductionConditions}
\begin{equation}\label{eq:ReproductionConditions1}
a^{(k)}(1) = m \tau (\tau - 1) \cdots (\tau - k + 1),\qquad k = 0, \ldots, l,
\end{equation}
which, by the following lemma, is equivalent to the power series $a(z) - mz^\tau$ having a zero of order $l+1$ at $z=1$, i.e.,
\begin{equation}\label{eq:ReproductionConditions2}
a(z) - mz^\tau = \sum_{k = l + 1}^\infty \left[\frac{a^{(k)}(1)}{k!} - m \frac{\tau (\tau - 1) \cdots (\tau - k + 1)}{k!} \right] (z-1)^k.
\end{equation}
\end{subequations}

\begin{lemma}\label{lem:reproduction}
For any Laurent polynomial $a(z)$ and associated shift $\tau$ as in \eqref{eq:tau}, the conditions \eqref{eq:ReproductionConditions1} and \eqref{eq:ReproductionConditions2} are equivalent.
\end{lemma}
\begin{proof}
Since $a(z)$ is a polynomial times a power of $z$, it admits a power series expansion
\[ a(z) = \sum_{k=0}^\infty \frac{a^{(k)}(1)}{k!} (z-1)^k.\]
Similarly, $mz^\tau$ can be expanded as a binomial series, 
\[ mz^\tau = m\big(1 + (z-1)\big)^\tau = m\sum_{k=0}^\infty \frac{\tau (\tau - 1) \cdots (\tau - k + 1)}{k!} (z-1)^k. \]
The difference of these expansions is
\begin{equation*}
a(z) - mz^\tau = \sum_{k = 0}^\infty \left[\frac{a^{(k)}(1)}{k!} - m \frac{\tau (\tau - 1) \cdots (\tau - k + 1)}{k!} \right] (z-1)^k,
\end{equation*}
which takes the form \eqref{eq:ReproductionConditions2} precisely when \eqref{eq:ReproductionConditions1} holds, in which case $a(z) - mz^\tau$ has a zero of order $l+1$ at $z = 1$.
\end{proof}

\subsection{Shifted schemes}
For a scheme $S_\bfa$ with mask $\bfa = [a_j]_j$ and $k\in \ZZ$, consider the \emph{shifted} scheme $S_{\overline{\bfa}}$ with mask $\overline{\bfa} = [a_{j+k}]_j$. Let us conclude this section describing how the above properties are affected by shifting.

If $[a_j]_j$ is odd (even) symmetric, then, for any integer $j_0$, the shifted mask $[a_{j + j_0}]_j$ is odd (even) symmetric as well. An integral shift in the index of the mask $\bfa$ corresponds to an integral shift in the argument of the limit function $f$; in particular its support size, regularity, and ability to generate polynomials are unchanged. If we define the corresponding shifts $\tau_\bfa$ and $\tau_{\overline{\bfa}}$ as in \eqref{eq:tau}, then
\begin{equation}\label{eq:shifted}
   \tau_{\overline{\bfa}}
 = \frac{\overline{a}'(1)}{m}
 = \frac{( z^k a(z) )'(1)}{m}
 = \frac{a'(1) + k a(1)}{m}
 = \frac{a'(1)}{m} + k
 = \tau_{\bfa} + k,
\end{equation}
where we used the convergence assumption \eqref{eq:convergence}. With these parameter shifts, the shifted scheme $S_{\overline{\bfa}}$ reproduces polynomials up to degree $l$ precisely when $S_\bfa$ reproduces polynomials up to degree $l$ \cite[Corollary 5.1]{Conti.Hormann11}.

\section{Pseudo-spline symbols}\label{sec:pseudo-splines}
\subsection{Definition and examples}

\noindent The conditions for polynomial generation and reproduction have been been applied in \cite{Conti.Hormann11} to define pseudo-splines of any arity, generalizing the families of primal and dual binary pseudo-splines described in \cite{Dong.Shen07} and \cite{Dyn.Hormann.Sabin.Shen08}.

\begin{definition}\label{def:pseudosplines}
For any $\tau\in \RR$ and $n,l \in \NN_0$ and integer $m\geq 2$, the $m$-ary \emph{pseudo-spline} with parameters $n,l$ and shift $\tau$ is defined to be the scheme $S_\bfa$ with minimal support satisfying
\begin{itemize}
\item[] Condition \eqref{eq:polynomialgeneration}, for polynomial generation up to degree $n$, and 
\item[] Condition \eqref{eq:ReproductionConditions}, necessary for polynomial reproduction up to degree $l$.
\end{itemize}
\end{definition}

\begin{remark}
Since polynomial reproduction of degree $l$ also requires Condition~\eqref{eq:polynomialgeneration} to hold with $n = l$, the actual degree of polynomial reproduction is $\min (n,l)$.
\end{remark}
For specific parameters $m,n,l$, and $\tau$, the symbol of the pseudo-spline scheme can be found by direct computation. Applying the Leibniz rule with respect to the factorization \eqref{eq:polynomialgeneration2}, Condition \eqref{eq:ReproductionConditions1} becomes
\[ m\sum_{j=0}^k {k\choose j} \frac{\rmd^{k-j} \sigma_m^{n+1} }{\rmd z^{k-j} } (1) \frac{\rmd^j b}{\rmd z^j}(1) = m\tau(\tau-1)\cdots(\tau-k+1),\qquad k = 0,\ldots,l, \]
forming a linear system $\bfA\bfd = \bfc$, with
\[\bfc = m[1,\tau,\ldots,\tau(\tau-1)\cdots(\tau-l+1)]^\rmT,\qquad \bfd = [b(1),b'(1),\ldots,b^{(l)}(1)]^\rmT, \]
and
\[ \bfA = m
\begin{bmatrix}
{0\choose 0} \sigma_m^{n+1}(1)   & 0 & \cdots & 0\\
{1\choose 0} \frac{\rmd \sigma_m^{n+1}}{\rmd z} (1)  & {1\choose 1} \sigma_m^{n+1}(1)  & \cdots & 0\\
\vdots & \vdots & \ddots & \vdots\\
{l\choose 0} \frac{\rmd^l \sigma_m^{n+1}}{\rmd z^l}(1) & {l\choose 1} \frac{\rmd^{l-1} \sigma_m^{n+1}}{\rmd z^{l-1}}(1) & \cdots & {l\choose l} \sigma_m^{n+1} (1)
\end{bmatrix}.
\]
Note that the entries of $\bfA$ can be computed recursively using the chain rule and \eqref{eq:sigmaderivative}, and explicitly using Fa\`a di Bruno's formula.

\begin{example}
Let $m = n = l = 3$. Setting up and solving the above system yields
\[
\bfA =
\begin{bmatrix}
  3 &   0 &  0 & 0\\
 12 &   3 &  0 & 0\\
 44 &  24 &  3 & 0\\
144 & 132 & 36 & 3
\end{bmatrix},\ 
\bfc = 3\begin{bmatrix} 1\\\tau\\ \tau(\tau-1)\\ \tau(\tau-1)(\tau-2) \end{bmatrix},\ 
\bfd = \begin{bmatrix} 1\\ \tau - 4\\ \tau^2 - 9\tau + 52/3\\ \tau^3 - 15\tau^2 + 66\tau - 80 \end{bmatrix}
\]
In particular for the primal parametrization, $\tau \equiv 0$ and one obtains the symbol of the primal ternary 4-point Dubuc-Deslauriers scheme \cite{Deslauriers.Dubuc89},
\[ a(z) = 3\left(\frac{1 + z + z^2}{3}\right)^4 \left(- \frac43 + \frac{11}{3}z - \frac43 z^2\right). \]
For the dual parametrization, $\tau \equiv 1/2$ and one obtains the symbol of the dual ternary 4-point Dubuc-Deslauriers scheme \cite{Ko.Lee.Yoon07},
\[ a(z) = -\frac{1}{16}\left(\frac{1 + z + z^2}{3}\right)^4\cdot (1 + z)\cdot (35 - 94z + 35z^2). \]
\end{example}

\begin{example}
Let $l = 1$ and suppose \eqref{eq:polynomialgeneration2} holds with $n\geq 1$. For $k = 0$, Condition \eqref{eq:ReproductionConditions1} becomes $b(1) = a(1)/m = 1$, which is a consequence of the convergence by \eqref{eq:convergence2}. For $k = 1$, Condition \eqref{eq:ReproductionConditions1} demands
\[ \tau = \frac{a'(1)}{m} = \frac12 (m-1)(n+1) + b'(1) \]
by \eqref{eq:sigmaderivative}. Replacing $b(z)$ by $z^k b(z)$ does not change the minimal support property and adds $k$ to the shift $\tau$ by \eqref{eq:shifted}. Up to an index shift of the mask, therefore, the value of $b'(1)$ is defined modulo the integers. Thus the size of the mask with minimal support depends on how $m, n$, and $\tau\pmod \ZZ$ combine.

As remarked in Section \ref{sec:background}, to reproduce linear polynomials, odd symmetric symbols $a(z)$ require $\tau = 0 \pmod \ZZ$, while even symmetric symbols require $\tau = 1/2\pmod \ZZ$. Therefore, we arrive at the following possibilities for the symbol $b(z)$ of minimal length, for some $k\in \ZZ$:
\begin{center}
\begin{tabular*}{\columnwidth}{l@{\extracolsep{\stretch{1}}}*{3}{c}@{}}
\toprule
                      & $\tau \equiv 0\pmod \ZZ$ & $\tau\equiv 1/2\pmod \ZZ$ \\ \midrule
 $m$ odd  or $n$ odd  &      $z^k$    &      $z^k(1+z)/2$  \\
$m$ even and $n$ even & $z^k(1+z)/2$  &          $z^k$     \\
\bottomrule
\end{tabular*}
\end{center}
For other values of $\tau\pmod \ZZ$ one obtains asymmetric schemes, which are beyond the scope of this paper.
\end{example}

\begin{remark}
For $m=2$ one has $\sigma_m(z) = (1+z)/2$. By Remark \ref{rem:symmetryfactorisation}, any even symmetric symbol has a factor $1 + z$, and after removing all factors $\sigma_m(z)$ from $a(z)$ and shifting we are left with an odd symmetric symbol
\[ b(z) = b_0 + b_1 (z^1 + z^{-1}) + \cdots + b_k (z^k + z^{-k}). \]
Then
\[ b'(z) = 1\cdot b_1 (1 - z^{-2}) + \cdots + k\cdot b_k (z^{k-1} - z^{-k-1}) \]
implies $b'(1) = 0$ and therefore $\tau = (n+1)/2$. Hence for the binary schemes of minimal length in the above example, only the diagonal case $b(z) = z^k$ occurs.
\end{remark}

\subsection{Generating function approach}
In this section we write $m = 2m' + \epsilon$, with $\epsilon \in \{0,1\}$. We assume that the parameter $l = 2l' + 1$ is odd and that $b(z)$ is odd symmetric and centered at zero, as in this case we are able to determine the regularity exactly.

How should we choose $b(z)$ such that $a(z) = m \sigma_m^{n+1}(z) b(z)$ satisfies \eqref{eq:ReproductionConditions1}? A convenient basis for the vector space of odd symmetric symbols whose support ranges from $-l'$ to $l'$ turns out to be
\[ \{1, \delta(z), \ldots, \delta^{l'}(z)\},\qquad \delta(z) := - \frac{(1-z)^2}{4z}.\]

For fixed integers $m\ge 2$ and $n\ge 0$, consider the generating function
\begin{equation}\label{eq:primalgenerating}
\GGG(y) = \GGG_{m,n} (y) := \left( \frac{m}{U_{m-1}\left(\sqrt{1-y}\right)} \right)^{n+1},
\end{equation}
where $U_{m-1}$ is the Chebyshev polynomial of the second kind of degree $m-1$ defined implicitly by
\begin{equation}\label{eq:Chebyshev1}
U_{m-1} (x) = \frac{\sin(m\theta)}{\sin(\theta)}, \qquad x = \cos(\theta),
\end{equation}
explicitly by
\begin{equation}\label{eq:Chebyshev2}
U_{m-1}(x) := \frac{ \left(x + \sqrt{x^2 -1}\right)^m - \left(x - \sqrt{x^2 -1}\right)^m }{2\sqrt{x^2 - 1}},
\end{equation}
or recursively by
\begin{equation}\label{eq:Chebyshev3}
U_0(x) = 1,\quad U_1(x) = 2x,\quad U_d(x) = 2x U_{d-1}(x) - U_{d-2}(x),\quad d\geq 2.
\end{equation}

\begin{lemma}\label{lem:generatingpositive}
For any $m\ge 2$ and $n\ge 0$, the generating function $\GGG$ admits, for positive coefficients $g_k$, the power series expansion
\begin{equation}\label{eq:Gpowerseries}
\GGG(y) = \sum_{k=0}^\infty g_k y^k.
\end{equation}
\end{lemma}

\begin{proof}
The Chebyshev polynomial $U_{m-1}(x)$ has roots $\cos\left(\frac{k\pi}{m}\right)$, $k = 1, \ldots$, $m-1$, by \eqref{eq:Chebyshev1} and leading coefficient $2^{m-1}$ and degree $m-1$ by \eqref{eq:Chebyshev3}, so that it admits the factorization
\[ U_{m-1} (x) = 2^{m-1} \prod_{k=1}^{m-1} \left[x - \cos\left(\frac{k\pi}{m} \right) \right]
           = 2^{m-1} x^{1 - \epsilon} \cdot \prod_{k=1}^{m' - 1 + \epsilon}\left[x^2 - \cos^2\left(\frac{k\pi}{m} \right) \right], \]
from which it follows that
\begin{equation}\label{eq:ChebFactor}
\frac{m}{U_{m-1}\left(\sqrt{1-y}\right)} = \frac{m}{2^{m-1}} \cdot \left(\frac{1}{1-y}\right)^{(1-\epsilon)/2} \cdot
\prod_{k=1}^{m'-1+\epsilon} \frac{1}{\sin^2\left(\frac{k\pi}{m} \right)- y}.
\end{equation}
Expanding into geometric series and taking powers one obtains the power series expansion \eqref{eq:Gpowerseries} with $g_k > 0$ for all~$k$.
\end{proof}

\begin{lemma} For any integers $m\geq 2$ and $n\geq 0$, the composition $\GGG\circ \delta$ is analytic at $z=1$.
\end{lemma}
\begin{proof}
The factorization \eqref{eq:ChebFactor} implies that
$(1 - y)^{(1-\epsilon)(n+1)/2} \cdot \GGG(y)$
is a rational function of $y$ with no pole at $y = 0$, so that substituting $y = \delta(z)$ gives a rational function of $z$ with no pole at $z = 1$, which is therefore analytic at $z = 1$. Since
\begin{equation}\label{eq:sqrt1mindelta}
\sqrt{1 - \delta(z)} = \sqrt{\frac{(1+z)^2}{4z}} = \frac{z^{-1/2} + z^{1/2}}{2}
\end{equation}
is analytic and nonzero at $z = 1$, its reciprocal is analytic at $z = 1$, implying that $\GGG\circ \delta$ is analytic at $z = 1$.
\end{proof}

Let $\GGG_{m,n,l}$ be the Taylor polynomial of degree $l' = (l - 1)/2$ to $\GGG_{m,n}$ at $y = 0$.

\begin{theorem}
The scheme with symbol
\begin{equation}\label{eq:pseudosymbol}
a_{m,n,l}(z) := m\sigma_m^{n+1}(z) b_{m,n,l}(z), \qquad b_{m,n,l}(z) := \GGG_{m,n,l} \big(\delta(z) \big)
\end{equation}
is a pseudo-spline of type $(n,l)$ with shift $\tau = (m-1)(n+1)/2$.
\end{theorem}
\begin{proof}
As the symbol \eqref{eq:pseudosymbol} satisfies the factorization \eqref{eq:polynomialgeneration2}, the scheme generates polynomials up to degree $n$. By \eqref{eq:Chebyshev2} and \eqref{eq:sqrt1mindelta},
\begin{align}\label{eq:ChebSine}
\!\!\!U_{m-1}\left(\sqrt{1-\delta(z)}\right)
& = \frac{ \left(\frac{1+z}{2\sqrt{z}} + \frac{1-z}{2\sqrt{z}}\right)^m - \left(\frac{1+z}{2\sqrt{z}} - \frac{1-z}{2\sqrt{z}}\right)^m }{2\frac{1-z}{2\sqrt{z}}}
 = \frac{ z^{-m/2} - z^{m/2} }{z^{-1/2} - z^{1/2}},
\end{align}
which gives
\[ \sigma_m^{n+1}(z) \cdot \GGG_{m,n}\big(\delta(z)\big)
= \left[\frac{1-z^m}{m(1-z)} \frac{m(z^{-1/2} - z^{1/2})}{z^{-m/2} - z^{m/2}}\right]^{n+1}
= z^\tau.\]
It follows that
\begin{align*}
a_{m,n,l}(z) - mz^\tau
& = -m \sigma_m^{n+1}(z)\big[ \GGG_{m,n}\big(\delta(z)\big) - \GGG_{m,n,l}\big(\delta(z)\big)\big]\\
& = -m \sigma_m^{n+1}(z) \delta^{l'+1}(z) \sum_{k=0}^\infty g_{k+l'+1} \delta^k(z).
\end{align*}
By Lemma \ref{lem:generatingpositive}, the coefficient $g_{l'+1} > 0$, so that the latter series is nonzero at $z = 1$, and $a_{m,n,l}(z) - mz^\tau$ has a zero of order exactly $2(l' + 1) = l+1$ at $z = 1$. Hence the theorem follows from Lemma \ref{lem:reproduction}.
\end{proof} 

\subsection{Explicit formulas for the symbol}
Next we determine explicit expressions for the pseudo-spline symbols in specific cases. First we need to bring the generating function $\GGG$ into a form convenient for computations.

\begin{theorem}\label{thm:ChebyshevExplicit}
For $m = 2m' + \epsilon$, with $\epsilon \in \{0,1\}$ and $m'\in \NN$, one has
\begin{equation}\label{eq:ChebyshevExplicit}
U_{m-1} \left(\sqrt{1-y}\right) = m\sqrt{1-y}^{1-\epsilon} \sum_{k = 0}^{m' - 1 + \epsilon} \frac{y^k}{(2k+1)!} \prod_{j=1}^k \left((2j - \epsilon)^2 - m^2\right) .
\end{equation}
\end{theorem}
\begin{proof}
From \eqref{eq:Chebyshev3} it follows that the statement holds for $m = 2, 3$. Suppose that $m = 2m' + \epsilon\geq 4$, with $\epsilon = 0$, and that the statement holds for all smaller $m$. For any $m'\in \NN$ and $k \in \NN_0$, the identities
\begin{align}
\frac{m'}{m'+k} \prod_{j=1}^k \left( 4j^2 - 4m'^2\right) & = \prod_{j=1}^k \left( (2j-1)^2 - (2m'-1)^2\right), \\
\frac{m'}{m'+k} \prod_{j=1}^k \left( 4j^2 - 4m'^2\right)& = \frac{m'-1}{m'-1-k}\prod_{j=1}^k \left( (2j)^2 - (2m'-2)^2\right)
\end{align}
follow by taking quotients, factoring the differences of the squares, and simplifying the telescoping product. Together with the hypothesis, we obtain
\begin{align*}
U_{m-2}\left(\sqrt{1-y} \right)
& = \sum_{k = 0}^{m' - 1} \frac{y^k}{(2k+1)!} (2m'-1) \prod_{j=1}^k \left( (2j-1)^2 - (2m'-1)^2 \right) \\
& = \sum_{k = 0}^{m' - 1} \frac{y^k}{(2k+1)!} (2m'-1) \frac{m'}{m'+k} \prod_{j=1}^k \left( 4j^2 - 4m'^2\right),\\
\frac{U_{m-3}\left(\sqrt{1-y} \right)}{\sqrt{1-y}}
& = \sum_{k=0}^{m'-2} \frac{y^k}{(2k+1)!} (2m'-2) \prod_{j=1}^k \left( (2j)^2 - (2m'-2)^2\right)\\
& = \sum_{k=0}^{m'-2} \frac{y^k}{(2k+1)!} \frac{2m'(m'-1-k)}{m'+k} \prod_{j=1}^k \left( 4j^2 - 4m'^2\right).
\end{align*}
Hence it follows from \eqref{eq:Chebyshev3} that
\begin{align*}
\frac{U_{m-1}\left(\sqrt{1-y}\right)}{\sqrt{1-y}} & = 2U_{m-2}\left(\sqrt{1-y}\right) - \frac{U_{m-3}\left(\sqrt{1-y}\right)}{\sqrt{1-y}}\\
& = 2m'\sum_{k=0}^{m'-1} \frac{y^k}{(2k+1)!} \prod_{j=1}^k \left(4j^2 - 4m'^2\right),
\end{align*}
which is equivalent to \eqref{eq:ChebyshevExplicit} under the assumption $\epsilon = 0$. The case $\epsilon = 1$ is derived analogously. By induction we conclude that \eqref{eq:ChebyshevExplicit} holds for all $m\geq 2$.
\end{proof}

The following corollary follows immediately from Theorem \ref{thm:ChebyshevExplicit} and the binomial series expansion.

\begin{corollary}
With
\[ P_m(y) := -\sum_{k = 1}^{m' + \epsilon - 1}
\left[\prod_{j = 1}^k \big((2j - \epsilon)^2 - m^2\big)\right]
\frac{y^k}{(2k+1)!}, \]
the generating function takes the form
\[ \GGG_{m,n}(y)
  = \left( \frac{(1-y)^{(\epsilon-1)/2}}{\displaystyle 1 - P_m(y)}\right)^{n+1}
  = (1-y)^{(n+1)(\epsilon-1)/2} \sum_{k = 0}^\infty {n + k \choose k} P_m^k(y). \]
\end{corollary}

\subsubsection{Binary pseudo-splines}
For $m = 2$ one has $P_2(y) = 0$ and
\[ \GGG_{2,n} (y) = (1 - y)^{-(n+1)/2},\]
and the binary pseudo-spline of type $(n,l)$ and shift $\tau = (n+1)/2$ has symbol
\begin{equation}\label{eq:SymbolBinary}
a_{2,n,l}(z) = 2\left(\frac{1 + z}{2} \right)^{n+1}\!\!\!\!\!\!\!\!\cdot b_{2,n,l}(z),\quad b_{2,n,l}(z) = \sum_{k = 0}^{l'} {n/2 - 1/2 + k \choose k} \delta^k(z).
\end{equation}
For odd $n$ one recovers the primal binary pseudo-splines, and for even $n$ the dual binary pseudo-splines; c.f.  \cite{Dyn.Hormann.Sabin.Shen08}.

\subsubsection{Ternary pseudo-splines}
For $m = 3$ one has $P_3(y) = \frac{4}{3} y$ and
\[ \GGG_{3,n}(y)
 =  \left( \frac{1}{1-\frac{4}{3}y} \right)^{n+1}
 = \sum_{k = 0}^\infty {n + k \choose k} \left(\frac{4}{3} y \right)^k. \]
Thus the ternary pseudo-spline of type $(n,l)$ with shift $\tau = n+1$ has symbol
\begin{equation}\label{eq:SymbolTernary}
a_{3,n,l}(z) = 3\left(\frac{1 + z + z^2}{3}\right)^{n+1} \!\!\!\!\!\!\!\!\!\cdot  b_{3,n,l}(z),\ \  b_{3,n,l}(z) = \sum_{k = 0}^{l'} {n + k \choose k} \left(\frac{4}{3} \delta(z) \right)^k \!\!\!.
\end{equation}

\subsubsection{Quaternary pseudo-splines}
For $m = 4$ one has $P_4(y) = 2y$ and, expanding as a binomial series and taking the Cauchy product,
\[ \GGG_{4,n} = (1-y)^{-(n+1)/2} \cdot \sum_{k = 0}^\infty {n+k \choose k} 2^k y^k = \sum_{k = 0}^\infty g_k y^k, \]
where
\[ g_k = \sum_{j = 0}^k {j + \frac{n-1}{2}\choose j} {n+k-j \choose k-j} 2^{k-j}. \]
Thus the quaternary pseudo-spline of type $(n,l)$ with shift $\tau = 3(n+1)/2$ has symbol
\begin{equation}\label{eq:SymbolQuaternary}
a_{4,n,l}(z) = 4\left(\frac{1 + z + z^2 + z^3}{4}\right)^{n+1}\cdot b_{4,n,l}(z),\quad b_{4,n,l}(z) = \sum_{k = 0}^{l'} g_k \delta^k (z) .
\end{equation}

\subsubsection{Small reproduction order $l$}\label{sec:symbol-small-l}
If $l' = 0$ then $b_{m,n,1} = 1$, and one recovers the well-known $m$-ary B-spline schemes. If $l' = 1$, then, modulo $\delta^2$,
\[ b_{m,n,3}
\equiv \left(1 + \frac12 (n+1)(1-\epsilon)\delta\right)\cdot \big(1 + (n+1) P_m(\delta) \big)
\equiv 1 + \frac16 (n+1)\left(m^2 - 1\right) \delta, \]
so that $b_{m,n,3}(z) = b_1 z^{-1} + b_0 + b_1 z$, with
\begin{equation}\label{eq:b0b1forl=3}
b_0 = 1 + \frac{\big(m^2 - 1 \big)(n+1)}{12},\qquad
b_1 = -\frac{\big(m^2 - 1 \big)(n+1)}{24}.
\end{equation}
Similarly it is straightforward to obtain explicit expressions for $b_{m,n,2l'+1}$, with $l' = 2,3,\ldots$

\subsubsection{$(2l'+2)$-point Dubuc-Deslauriers schemes}
For $l'\in \NN_0$ and $l = 2l' + 1$, the primal $m$-ary $(l+1)$-point Dubuc-Deslauriers scheme is the $m$-ary pseudo-spline of type $(2l' + 1, 2l' + 1)$. From the generating function, we obtain $b_{m,1,1}(z) = 1$ and
\begin{align}
b_{m,3,3} (z) & = 1 + \frac{2}{3}(m^2 - 1) \delta(z) 
                = 1 + \frac13(m^2 - 1) - \frac16 (m^2 - 1)(z + z^{-1}) \label{eq:symbol-m-ary-4point}\\
b_{m,5,5} (z) & = 1 + (m^2 - 1) \delta(z) + \frac{2}{15} (4m^4 - 5m^2 + 1)\delta^2(z)
\end{align}

\begin{remark}
Explicit expressions for the symbol of the general primal and dual $m$-ary $(2l' + 2)$-point Dubuc-Deslauriers scheme can be obtained by evaluating the Lagrange interpolant of degree $l$ to uniform data at appropriate points. To make this precise, consider the Lagrange basis polynomials
\[ L^{2l' + 1}_j(x) := \prod_{\substack{k = -l'\\k\neq j}}^{l'+1} \frac{x - k}{j - k},\qquad
   L^{2l'    }_j(x) := \prod_{\substack{k = -l'\\k\neq j}}^{l'  } \frac{x - k}{j - k}. \]
By definition of their refinement rules, the primal and dual $m$-ary $(2l' + 2)$-point Dubuc-Deslauriers scheme have centered symbols
\[
1 + \sum_{s = 1}^{m - 1} z^s \sum_{j = -l'}^{l' + 1} L^l_j\left(\frac{s}{m}\right) z^{-mj} \qquad \text{and}\qquad
\sum_{s = 0}^{m-1} z^s \sum_{j = -l'}^{l' + 1} L^l_j \left(\frac{2s + 1}{2m_0}\right) z^{-mj}.
\]
\end{remark}

\begin{remark}
Although in this paper we only consider pseudo-splines for which the derived symbol $b(z)$ is odd symmetric, we conjecture that dual $m$-ary $(2l' + 2)$-point Dubuc-Deslauriers schemes have symbol $a(z) = m\sigma_m^n(z) b(z)$, where $n = l = 2l' + 1$ and $b(z)$ is $\frac{1 + z}{2}$ multiplied by the Taylor polynomial of degree $l'$ to the generating function
\begin{equation}\label{eq:generatingeven}
\GGG_{m,n}(y)/\sqrt{1-y}
\end{equation}
evaluated at $\delta(z)$. This is verified symbolically in the worksheet for small $l$.
\end{remark}

\subsubsection{$(2l'+1)$-point interpolatory schemes}
For $l' \in \NN$ and odd arity $m = 2m' + 1$, consider the family \cite{Lian09} of $m$-ary $(2l'+1)$-point interpolatory schemes with symbol
\[
a(z) = \sum_{k = -m l' -m'}^{ml' + m'} a_k z^k, \qquad
 a_k = a_{-k} = \frac{\prod_{i = 1}^{l'-j} (im + k) \prod_{i = 1}^{l'+j} (im - k)}{m^{2l'}(l'-j)! (l'+j)!},
\]
for $k = 0, \ldots , m'$ and $j = 0$, or $k = mj - m', \ldots, mj + m'$ and $j = 1, \ldots, l'$. These are the $m$-ary pseudo-splines of type $(2l', 2l'+1)$ and shift $\tau = 0$. 

The family of ternary schemes in \cite{Zheng.Hu.Peng09}, depending on a parameter $u$, is the affine span of the above ternary $(2l'+1)$-point scheme (with $u = L^{2l'}_{-l'}\left(-1/3\right)$), for which it attains maximal order of reproduction, and the ternary $2l'$-point Dubuc-Deslauriers schemes (with $u = -L_{l'}^{2l'-1} \left(-1/3\right)$). Even more specifically, for $l' = 1$ one obtains the original scheme from \cite{Hassan.Dodgson02}.

\section{Regularity}\label{sec:regularity}
Let be given the scheme \eqref{eq:scheme} with symbol, after shifting, satisfying the conditions \eqref{eq:polynomialgeneration2} of polynomial generation of degree $r\geq 0$. In this equation, suppose that $b(z)$ is centered at zero and odd symmetric, so that its Fourier transform takes the form
\[ B(\xi) := b(\rme^{-i\xi}) = b_0 + 2\sum_{j=1}^p b_j \cos(j\xi), \qquad \xi\in \RR,\]
for some $p\geq 1$. 

\subsection{A recipe for computing the exact regularity}\label{sec:recipe}
Let $\bfM$ be the matrix defined by
\begin{equation}\label{eq:Mfolded}
\bfM = [m_{j,k}]_{j,k=0,\ldots,\lfloor \frac{p-1}{m-1} \rfloor},\qquad
m_{j,k} := \left\{
\begin{array}{ll}
b_j & k = 0,\\
b_{|j-mk|} + b_{j+mk} & k\geq 1,
\end{array}
\right.
\end{equation}
which can be interpreted as a `folded' submatrix of the subdivision matrix. See Table \ref{tab:M} for explicit expressions for $\bfM$ for various $m\geq 2$ and $p\geq 1$.

\begin{table}
\begin{tabular*}{\columnwidth}{@{\extracolsep{\stretch{1}}}*{4}{c}}
\toprule
$\bfM$  & $m = 2$ & $m = 3$ & $m = 4$
\\ \midrule \vspace{0.25em}
$p = 1$ & $[b_0]$ & $[b_0]$ & $[b_0]$
\\ \vspace{0.25em}
$p = 2$ & $\begin{bmatrix}b_0&2b_2\\b_1&b_1 \end{bmatrix}$ &
 $[b_0]$ & $[b_0]$
 \\ \vspace{0.25em}
$p = 3$ &
$\begin{bmatrix}b_0 & 2b_2 & 0\\ b_1 & b_1 + b_3 & b_3\\ b_2 & b_0 & b_2\end{bmatrix}$ & $\begin{bmatrix} b_0 & 2b_3\\ b_1 & b_2\end{bmatrix}$ & $[b_0]$
\\ \vspace{0.25em}
$p = 4$ &
$\begin{bmatrix}
b_0 & 2b_2 & 2b_4 & 0\\
b_1 & b_1 + b_3 & b_3 & 0\\
b_2 & b_0 + b_4 & b_2 & b_4\\
b_3 & b_1 & b_1 & b_3
\end{bmatrix}$ &
$\begin{bmatrix} b_0 & 2b_3\\ b_1 & b_2 + b_4\end{bmatrix}$
& $\begin{bmatrix} b_0 & 2b_4\\ b_1 & b_3\end{bmatrix}$	
\\ \vspace{0.25em}
$p = 5$ & 
$\begin{bmatrix}
b_0 &      2b_2 &      2b_4 &       0 &       0\\
b_1 & b_1 + b_3 & b_3 + b_5 &     b_5 &       0\\
b_2 & b_0 + b_4 &       b_2 &     b_4 &       0\\
b_3 & b_1 + b_5 &       b_1 &     b_3 &     b_5\\
b_4 &       b_2 &       b_0 &     b_2 &     b_4\\
\end{bmatrix}$ &
$\begin{bmatrix} b_0 & 2b_3 & 0 \\ b_1 & b_2 + b_4 & b_5\\ b_2 & b_1 + b_5 & b_4\end{bmatrix}$
& $\begin{bmatrix}b_0 & 2b_4\\ b_1 & b_3 + b_5 \end{bmatrix}$
\\
\bottomrule
\end{tabular*}

\caption{The matrix $\bfM$ in \eqref{eq:Mfolded} for various $m\geq 2$ and $p\geq 1$.}\label{tab:M}
\end{table}

Under the assumptions, the following theorem yields a quick method for computing the exact regularity of the scheme \eqref{eq:scheme}.

\begin{theorem}\label{thm:MainTheorem}
Suppose the matrix $\bfM$ has spectral radius $\rho > 1/m$.
\begin{enumerate}
\item If $B(\xi) \geq 0$ for all $\xi$, then the regularity of the scheme \eqref{eq:scheme} has lower bound
\[ r - \log_m(\rho).\]
\item If $B(\xi) > 0$ for all $\xi$, then this bound is optimal.
\end{enumerate}
\end{theorem}

The proof is a straightforward but lengthy generalization of the binary case presented in the report \cite{Floater.Muntingh13}; for part a. see Theorem \ref{thm:MainTheoremLower} and for part b. see Theorem \ref{thm:MainTheoremStrict} in Appendix \ref{sec:AppendixB}.

\begin{remark}
In practice one takes $r$ maximal in the factorization \eqref{eq:polynomialgeneration2}, so that $\bfM$ has minimal size.
\end{remark}

\subsection{Regularity of pseudo-splines}
Now consider the pseudo-spline scheme defined by \eqref{eq:pseudosymbol}. By Lemma \ref{lem:generatingpositive} and since $\delta(\rme^{-i\xi}) = \sin^2(\xi/2) \geq 0$, it follows that the derived symbol $b_{m,n,l}(z)$ in \eqref{eq:pseudosymbol} satisfies $b_{m,n,l}\big(\rme^{-i\xi}\big) = B_{m,n,l}(\xi) > 0$ for all $\xi$. Hence, whenever the matrix $\bfM$ has spectral radius $\rho > 1/m$, the scheme has exact regularity $r - \log_m (\rho)$.

In the next sections we compute these regularities explicitly for some special cases. In each case, the regularity is computed in a fraction of a second, while a similar numerical computation based on the Joint Spectral Radius can take a significant amount of time (unless more sophisticated techniques are applied, as in \cite{Moeller15}).

\begin{table}
\begin{tabular*}{\columnwidth}{@{ }@{\extracolsep{\stretch{1}}}*{10}{lc|cccccccc}@{ }}
\toprule
&  & & $l'=0$ & & $l'=1$ & & $l'=2$ & & $l'=3$ \\
\midrule
$ m = 2$
& $n = 1$ & \ps210 & \phantom{0}1 \\
& $n = 2$ & \ps220 & \phantom{0}2 & \ps221 & 1.19265\\
& $n = 3$ & \ps230 & \phantom{0}3 & \ps231 & 2\phantom{.00000}\\
& $n = 4$ & \ps240 & \phantom{0}4 & \ps241 & 2.83007 & \ps242 & 2.10558\\
& $n = 5$ & \ps250 & \phantom{0}5 & \ps251 & 3.67807 & \ps252 & 2.83007\\
& $n = 6$ & \ps260 & \phantom{0}6 & \ps261 & 4.54057 & \ps262 & 3.57723 & \ps263 & 2.87602\\
& $n = 7$ & \ps270 & \phantom{0}7 & \ps271 & 5.41504 & \ps272 & 4.34379 & \ps273 & 3.55113\\
\midrule
$m = 3$
& $n = 1$ & \ps310 & \phantom{0}1\\
& $n = 2$ & \ps320 & \phantom{0}2 & \ps321 & 1\phantom{.00000}\\
& $n = 3$ & \ps330 & \phantom{0}3 & \ps331 & 1.81734 \\
& $n = 4$ & \ps340 & \phantom{0}4 & \ps341 & 2.66528 & \ps342 & 1.57641 \\
& $n = 5$ & \ps350 & \phantom{0}5 & \ps351 & 3.53503 & \ps352 & 2.31986 \\
& $n = 6$ & \ps360 & \phantom{0}6 & \ps361 & 4.42110 & \ps362 & 3.09466 & \ps363 & 1.88409\\
& $n = 7$ & \ps370 & \phantom{0}7 & \ps371 & 5.31986 & \ps372 & 3.89404 & \ps373 & 2.58999 \\
\midrule
$m = 4$
& $n = 1$ & \ps410 & \phantom{0}1\\
& $n = 2$ & \ps420 & \phantom{0}2 & \ps421 & 0.87604\\
& $n = 3$ & \ps430 & \phantom{0}3 & \ps431 & 1.70752\\
& $n = 4$ & \ps440 & \phantom{0}4 & \ps441 & 2.57101 & \ps442 & 1.32536\\
& $n = 5$ & \ps450 & \phantom{0}5 & \ps451 & 3.45627 & \ps452 & 2.09955\\
& $n = 6$ & \ps460 & \phantom{0}6 & \ps461 & 4.35730 & \ps462 & 2.90432 & \ps463 & 1.60191\\
& $n = 7$ & \ps470 & \phantom{0}7 & \ps471 & 5.27028 & \ps472 & 3.73236 & \ps473 & 2.35154\\
\bottomrule
\end{tabular*}
\caption{Cardinal limit functions and regularities, rounded to five decimals, of binary, ternary, and quaternary pseudo-splines of type ($n, 2l'+1)$ and shift $\tau = (m-1)(n+1)/2$.}\label{tab:m-aryps}
\end{table}

\subsubsection{Binary pseudo-splines}
Using the explicit expression \eqref{eq:SymbolBinary} for the symbol of the binary pseudo-spline scheme, we apply the method in Section \ref{sec:recipe} to compute its regularity.

Table \ref{tab:m-aryps} shows the regularity of binary primal ($n$ odd) and dual ($n$ even) pseudo-splines \cite{Dyn.Hormann.Sabin.Shen08}. The first column corresponds to the binary B-spline scheme of degree $n$. The top slanted diagonal $n = 2l'$ corresponds to the $(2l' + 1)$-point scheme from \cite{Lian09}. Below that, the slanted diagonal $n = 2l' + 1$ (resp. $n = 2'l' + 2$) corresponds to the primal (resp. dual) ($2l'+2$)-point Dubuc-Deslauriers scheme \cite{Dubuc86, Deslauriers.Dubuc89, Dyn.Floater.Hormann04}.

For $n \geq 2l' + 1$, these exact regularities agree (up to three decimals) with the lower bounds presented in Tables 2 and 3 of \cite{Dong.Dyn.Hormann10}, established numerically using a Joint Spectral Radius computation. The regularity for the binary 3-point scheme, obtained by taking $(n,2l'+1) = (2,3)$, agrees with the exact Joint Spectral Radius computation in \cite[\S 7.2]{Moeller15}.

\subsubsection{Ternary pseudo-splines}
Using the explicit expression \eqref{eq:SymbolTernary} for the symbol of the ternary pseudo-spline scheme, we again apply the method in Section \ref{sec:recipe} to compute its regularity, shown in the second part of Table \ref{tab:m-aryps}. 

The first column corresponds to the ternary B-spline scheme of degree $n$ (cf. \cite{Hassan.Dodgson02} for $(n,l) = (3,1)$). The slanted diagonal corresponds to the primal ternary ($2l'+2$)-point Dubuc-Deslauriers scheme \cite{Deslauriers.Dubuc89}. 

The regularities for the primal ternary 4-point ($n = l = 3$) and 6-point ($n=l=5$) Dubuc-Deslauriers schemes agree with the lower bounds in \cite[Table 4.2]{Hassan05}.

\subsubsection{Quaternary pseudo-splines}
Using the explicit expression \eqref{eq:SymbolQuaternary} for the symbol of the quaternary pseudo-spline scheme, we again apply the method in Section \ref{sec:recipe} to compute its regularity, shown in the third part of Table \ref{tab:m-aryps}.

The first column corresponds to the quaternary B-spline scheme of degree $n$. The slanted diagonal corresponds to the primal quaternary ($2l'+2$)-point Dubuc-Deslauriers scheme \cite{Deslauriers.Dubuc89}. 

The regularity of the quaternary 3-point scheme ($m=4$, $n=2$ and $l=3$) agrees with the exact Joint Spectral Radius computation in \cite[\S 7.8]{Moeller15}.

\subsubsection{Small reproduction order $l$}
If $l' = 0$ one recovers the well-known fact that the $m$-ary B-spline scheme of degree $n$ has regularity $n$ (i.e., the B-spline of degree $n$ has regularity $n - \varepsilon$ for any $\varepsilon > 0$).

If $l' = 1$ and $n\geq 2$, then the folded matrix $\bfM$ has size $\lfloor \frac{l' - 1}{m - 1} \rfloor + 1 = 1$, and the regularity of the scheme can be expressed in terms of the central coefficient $b_0$ from \eqref{eq:b0b1forl=3} as
\begin{equation}\label{eq:lp1}
n - \log_m(b_0) = n - 2 - \log_m\left( \frac{1}{m^2} + \frac{n+1}{12}\left(1 - \frac{1}{m^2}\right) \right).
\end{equation}
In the limit $m\to \infty$, the regularity approaches $n - 2$. Moreover, the regularity decreases with $m$ for $n < 11$, increases with $m$ for $n > 11$, and stays constant at $9$ for $n = 11$. 

Next, consider any column $l'\geq 1$, with $m>l'$ and $n\geq 2l'$. Since $m > l'$, the folded matrix has dimension $\lfloor \frac{l'-1}{m-1} \rfloor + 1 = 1$, and the scheme has regularity $n - \log_m(b_0)$, with $b_0 z^0$ the constant term of $\GGG_{m, n}\big(\delta(z)\big) \mod \delta^{l'+1}$. Each monomial $\delta^k$ of $P_m(\delta)$ has as a coefficient a polynomial of degree $2k$ in $m$. It follows that $b_0$ is a polynomial of degree $2l'$ in $m$, and the regularity approaches
\[ n-\lim_{m\to \infty} \log_m(b_0) = n-2l' - \lim_{m\to \infty} \log_m(m^{-2l'} b_0) = n - 2l'\]
in the limit $m\to \infty$.

\begin{figure}
\includegraphics[scale = 0.48]{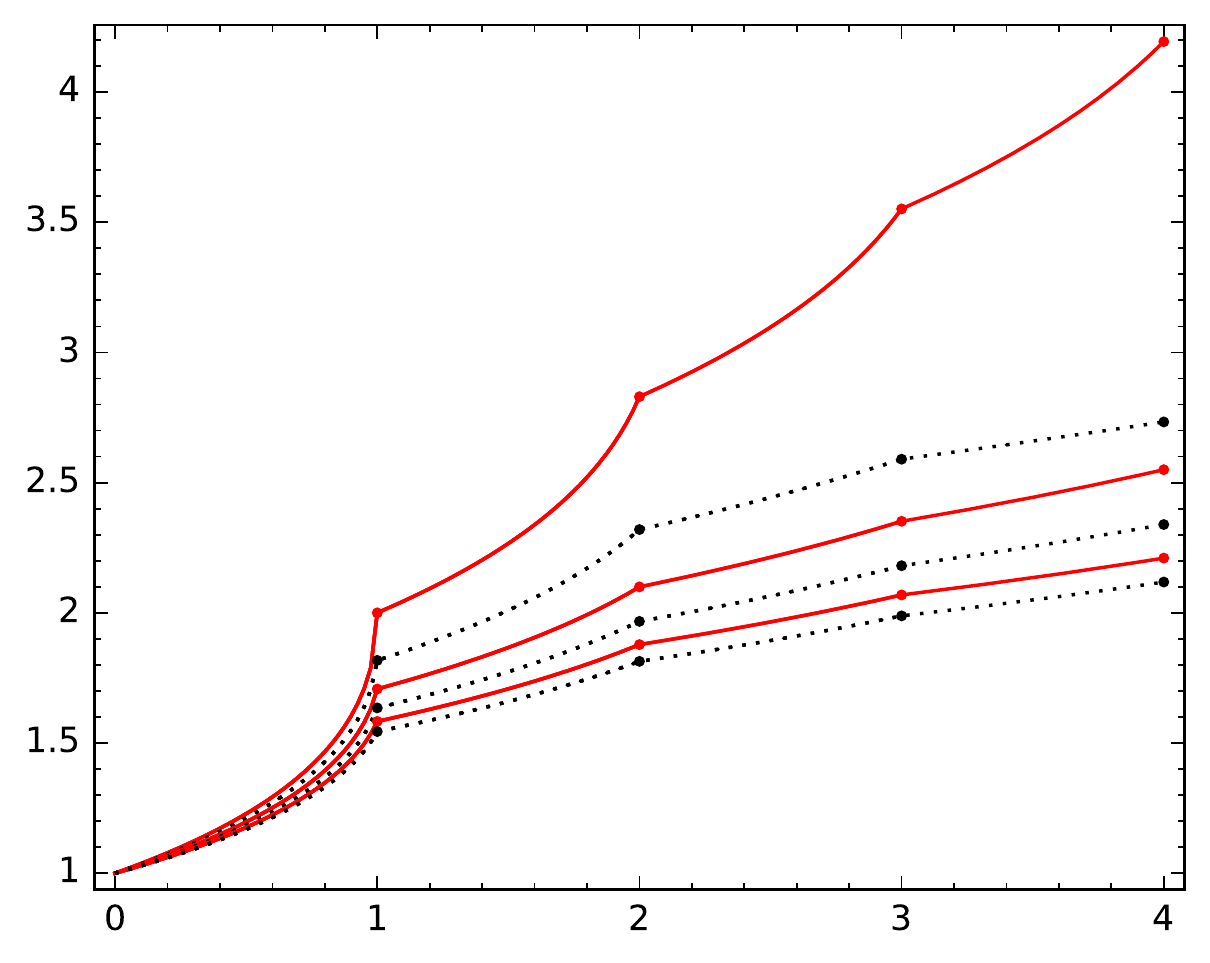}
\includegraphics[scale = 0.48]{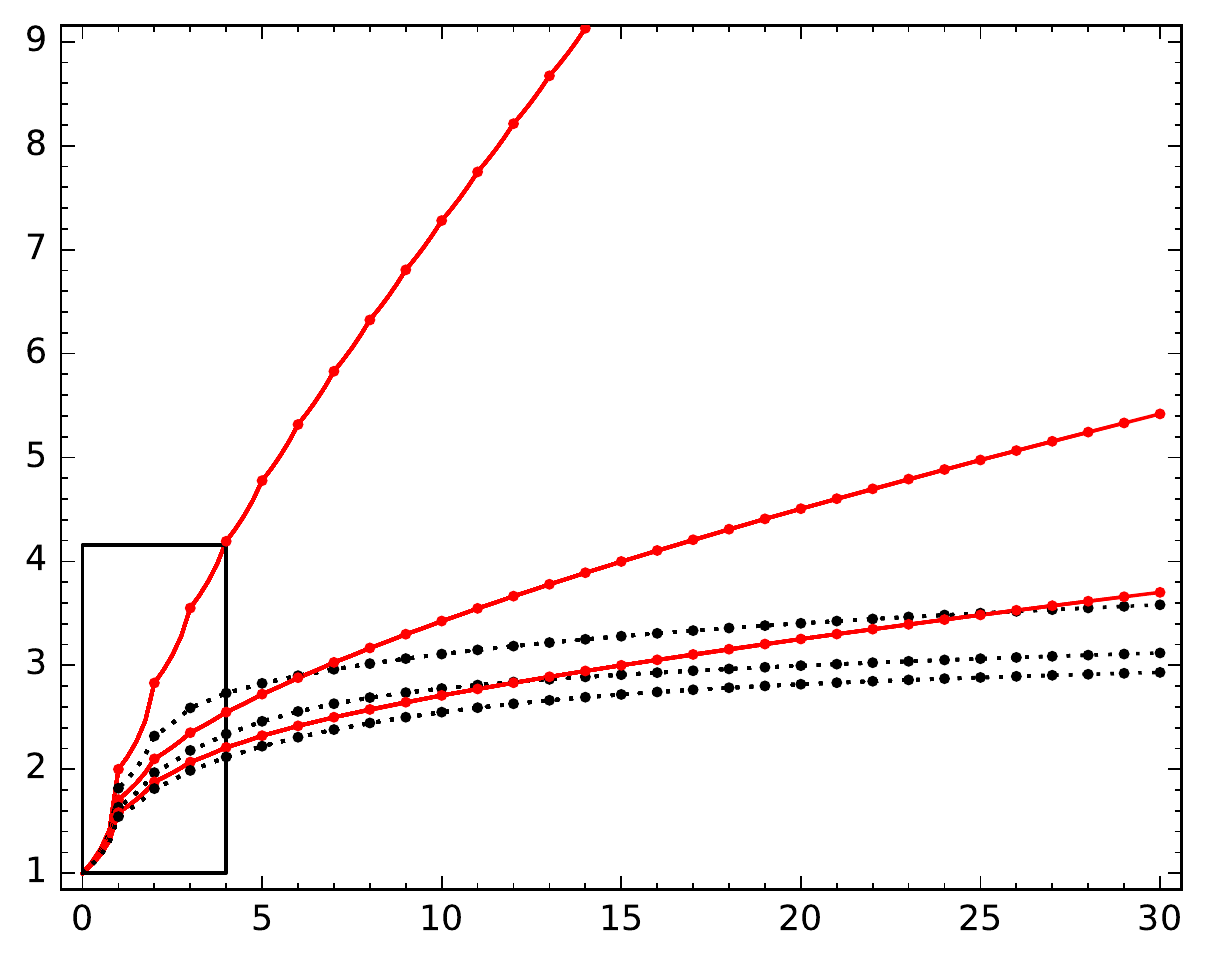}
\caption{For $m = 2m' + \varepsilon = 2, \ldots, 7$ and $l' = 0, \ldots, 4$ (left) and $l' = 0, \ldots, 30$ (right), the regularity of primal $2m'$-ary (solid) and $(2m'+1)$-ary (dotted) $(2l'+2)$-point Dubuc-Deslauriers schemes versus $l'$, interpolated by the regularities of successive tension parameter schemes.}\label{fig:RegularityDD}
\end{figure}

\subsubsection{Primal Dubuc-Deslauriers with tension}
By the linearity of the Fourier transform, any convex combination of masks with positive Fourier transform has positive Fourier transform. Thus Theorem~\ref{thm:MainTheorem} applies to convex combinations of the $m$-ary $2l'$-point and $(2l'+2)$-point primal Dubuc-Deslauriers schemes.

Because the method in Section \ref{sec:recipe} is very fast, it is possible to quickly compute the regularity for a large number of schemes of large size. Figure \ref{fig:RegularityDD} shows the regularity of these schemes for $m = 2,\ldots,7$. In the domain $l' = 0, \ldots, 4$, the left figure indicates that the regularity decreases (pointwise) when the arity increases. However, in the domain $l' = 0, \ldots, 30$ the right figure indicates a different picture, with the even arity (drawn solid) approaching asymptotically a steeper slope than the schemes with odd arity (drawn dotted).

For binary schemes, the slope of the top curve approaches the known value of $2 - \log_2(3)\approx 0.415$ \cite{Dong.Shen07}. Although affine, non-convex combinations of masks with positive Fourier transform do not necessarily have positive Fourier transform, in which case the method of Section~\ref{sec:recipe} is no longer valid, the regularity of such schemes has been analysed by other means \cite{Hechler.Moessner.Reif09} \cite[\S 7.3]{Moeller15}.

\begin{example}
A particular case is the classical 4-point scheme with tension~\cite{Dyn.Levin.Gregory87}. More generally, consider the $m$-ary 4-point scheme with tension with symbol
\[ a_\omega(z) := (1-\omega)a_{m,1,1}(z) + \omega a_{m,3,3}(z) = m\sigma_m^2(z) b_\omega(z),\qquad \omega \in [0,1], \]
where, using \eqref{eq:b0b1forl=3} and shifting to a centered symbol, 
\begin{align*}
b_\omega(z) 
       & = (1 - \omega) + \omega \frac{(1 + \cdots + z^{m-1})^2}{m^2 z^{m-1}} \left[ 1 + \frac{m^2 - 1}{3} - \frac{m^2 - 1}{6} \left(z^{-1} + z^{+1}\right)\right].
\end{align*}
Since $p = m$, a calculation yields the folded matrix
\[\bfM
= \begin{bmatrix} b_0 & 2b_m\\ b_1 & b_{m-1} \end{bmatrix}
= \frac{1}{m^2}\begin{bmatrix} m^2 - \frac13 \omega(m-1)(2m-1) & - \frac13 \omega (m^2 - 1)\\ \omega (m-1) & \omega \end{bmatrix}.
\]
It follows that, with
\[ D := (4m^4 - 24m^3 + 37m^2 - 12m + 4)\omega^2 - 6(2m^4 - 3m^3 + 4m^2)\omega + 9m^4, \]
the scheme has regularity
\[ 1 - \log_m\big(\rho(\bfM)\big),\qquad \rho(\bfM) =  \frac{3m^2 - 2m^2\omega + 3m\omega + 2\omega + \sqrt{D}}{6m^2}, \]
which is plotted in Figure \ref{fig:RegularityDD} from $l' = 0$ to $l' = 1$.
\end{example}

\section{Conclusion}\label{sec:conclusion}
Using a generating function approach, we have derived the symbol of the symmetric $m$-ary pseudo-spline of type $(n, 2l' + 1)$ and shift $\tau = (m-1)(n+1)/2$. It was shown how various schemes in the literature appear as special cases. For such pseudo-spline schemes, the derived mask is odd symmetric and has positive Fourier transform, making it possible to compute the exact regularity rapidly in terms of the spectral radius of a matrix.

In the future it would be interesting to show that the generating function \eqref{eq:generatingeven} can be used to define pseudo-splines with even symmetric derived symbol $b(z)$. An open question is whether it is possible to determine the regularity exactly for such schemes, which seems to be a very hard problem. Finally it remains to be seen to what length these results can be generalized to non-symmetric pseudo-splines.

\section*{Acknowledgments}
I am grateful to Maria Charina and Michael Floater for the many discussions on the topic of this paper. This projected was supported by a FRINATEK grant, project number 222335, from the Research Council of Norway.


\appendix 

\newpage
\section{Regularity of $m$-ary subdivision}\label{sec:AppendixB}

It is well known \cite{Rioul92, Han02} that for symmetric interpolatory schemes with positive Fourier transform, it is possible to determine the H\"older regularity exactly. In the report \cite{Floater.Muntingh13} it was shown that this is possible for non-interpolatory binary schemes as well. In this appendix we show that these results generalize to the general $m$-ary scheme \eqref{eq:scheme} (cf. \cite{Hassan05} for the ternary case). This is related to results described in \cite{Charina14, Moeller15, Moeller.Reif14}, which show that the underlying mathematical reason for the correctness of the method is the validity of the finiteness conjecture for the joint spectral radius of subdivision submatrices derived from schemes with positive Fourier transform.

In this appendix we suppose that $a(z)$ satisfies the conditions \eqref{eq:polynomialgeneration2} for polynomial generation up to some degree $r\ge 0$, and, after shifting the coefficients $a_k$ as necessary, that the mask $\bfb = (b_j)_j$ corresponding to $b(z)$ is odd symmetric and centered at zero, i.e.,
\begin{equation}\label{eq:bform}
 \bfb = [b_p, \ldots, b_1, b_0, b_1,\ldots, b_p],
  \qquad b_p\neq 0,
\end{equation}
for some $p \ge 0$. Then the Fourier transform of $\bfb$,
\[ B(\xi) := b(\rme^{-i\xi}) = b_0 + 2 \sum_{j = 1}^p b_j \cos(j\xi), \qquad \xi \in \RR, \]
is real and periodic with period $2\pi$.

\subsection{Regularity as a decay rate of differences of the data}\label{sec:reg}

The regularity of the limit function $f$ is related to the decay rate of divided differences of the scheme. For each integer $s \ge 0$,
let $f_{\ell,j}^{[s]}$ denote the divided difference of the values $f_{\ell,j-s},\ldots,f_{\ell,j}$ at the corresponding $m$-adic points $m^{-\ell}(j-s),\ldots,m^{-\ell}j$. That is,
\begin{equation}\label{eq:recurse}
 f_{\ell,j}^{[0]} = f_{\ell,j},\qquad  f_{\ell,j}^{[s]} = \frac{m^\ell}{s} \left(f_{\ell,j}^{[s-1]} - f_{\ell,j-1}^{[s-1]}\right), \qquad s\ge 1.
\end{equation}
Under condition \eqref{eq:polynomialgeneration},
there is a scheme for the $f_{\ell,j}^{[s]}$ for $s=0,\ldots,r+1$. Writing
\[ a^{[s]}(z) = \sum_j a_j^{[s]} z^j := \frac{a(z)}{\sigma_m^s(z)},\qquad f_\ell^{[s]}(z) := \sum_j f_{\ell,j}^{[s]} z^j, \]
this scheme takes the equivalent forms
\begin{equation}\label{eq:derivedschemea}
f_{\ell + 1, j}^{[s]} = \sum_k a_{j-mk}^{[s]} f_{\ell,k}^{[s]}, \qquad f_{\ell+1}^{[s]}(z) = a^{[s]}(z) f_\ell^{[s]}(z^m).
\end{equation}

Consider the differences $g_{\ell,j}^{[r]}$ (of the divided differences) of the data and the corresponding symbol, defined by
\[ g_{\ell,j}^{[r]} := f_{\ell,j}^{[r]} - f_{\ell,j-1}^{[r]},\qquad g_\ell^{[r]}(z) := \sum_j g_{\ell,j}^{[r]} z^j. \]
The following lemma relates the decay rate of $g_{\ell,j}^{[r]}$ to the regularity of the limit function $f$. It was shown to hold for binary schemes in \cite[Theorem 4.9]{Dyn.Levin02} and for ternary schemes in \cite[Theorem 3.4.4]{Hassan05}, but also holds for schemes with general arity $m$.

\begin{lemma}
Suppose that, for large enough $\ell$,
\begin{equation}\label{eq:growth}
|g^{[r]}_{\ell,j}| \le K\lambda^\ell,
\end{equation}
for some constants $K$ and $\lambda < 1$. Then $f^{(r)} \in C^0$. Moreover, if $1/m < \lambda < 1$, then $f^{(r)} \in C^{-\log_m(\lambda)}$.
\end{lemma}

\begin{proof}
To simplify notation, let us drop the superscripts in $g^{[r]}_{\ell,j}, g^{[r]}_\ell(z), f^{[r]}_{\ell,j}$, $f^{[r]}_\ell (z)$, $f^{(r)}, a^{[r]}(z)$. Using the standard parametrization, let $L_\ell$ denote the piecewise linear function through the points $(m^{-\ell}j, f_{\ell,j})$ at level $\ell$. We first bound the maximal difference between these piecewise linear functions at levels $\ell$ and $\ell+1$ in terms of the differences at level $\ell$. Since this maximum is attained at one of the breakpoints,
\begin{equation}\label{eq:diffpolygons}
\|L_{\ell+1} - L_\ell\|_\infty = \max_j \left|f_{\ell+1,j} - h_{\ell+1,j} \right|,
\end{equation}
where, writing $j = mj' + \varepsilon$,
\[ h_{\ell+1, mj' + \varepsilon} := \frac{m-\varepsilon}{m} f_{\ell,j'} + \frac{\varepsilon}{m} f_{\ell,j'+1},\qquad \varepsilon = 0, 1, \ldots, m - 1,\qquad j'\in \ZZ, \]
with corresponding symbol
\[ h_{\ell+1}(z) := \sum_j h_{\ell+1, j} z^j = \frac{(1 + z + \cdots + z^{m-1})^2}{mz^{m-1}} f_\ell(z^m).\]
Therefore
\begin{align*}
f_{\ell+1}(z) - h_{\ell+1}(z)
& = (1 + z + \cdots + z^{m-1}) d(z) f_\ell(z^m),
\end{align*}
with
\[ d(z) := \frac{a(z)}{1 + z + \cdots + z^{m-1}} - \frac{1 + z + \cdots + z^{m-1}}{mz^{m-1}}. \]
But $d(1) = a(1)/m - 1 = 0$ by \eqref{eq:convergence2}, so that $d(z) = (1 - z)e(z)$, with $e(z) = \sum_j e_j z^j$
a Laurent polynomial. Therefore
\[ f_{\ell+1}(z) - h_{\ell+1}(z) = e(z)(1 - z^m) f_\ell(z^m) = e(z) g_\ell(z^m), \]
or equivalently
\[ f_{\ell+1,j} - h_{\ell+1,j} = \sum_k e_{j - mk} g_{\ell,k}. \]
Using \eqref{eq:diffpolygons}, we obtain, for some constant $K_1$,
\[ \|L_{\ell+1} - L_\ell\|_\infty
  =  \max_j | f_{\ell+1,j} - h_{\ell+1,j} |
\leq \max_j \sum_k |e_{j - mk}| \cdot \max_k |g_{\ell,k}|
\leq K_1 \lambda^\ell, \]
from which it follows that $[L_\ell]_\ell$ is a Cauchy sequence. Equipped with the infinity norm $\|\cdot\|_\infty$, the space of bounded continuous functions on the real line is complete, and $[L_\ell]_\ell$ converges uniformly to a continuous function $f$. Moreover,
\begin{equation}\label{eq:difflimitpol}
\|f - L_\ell\|_\infty
\leq \sum_{\ell' = \ell}^\infty \|L_{\ell'+1} - L_{\ell'}\|_\infty
\leq \sum_{\ell' = \ell}^\infty K_1\lambda^{\ell'}
= K_2 \lambda^\ell,\quad K_2 := \frac{K_1}{1-\lambda},
\end{equation}
so that $[L_\ell]_\ell$ converges to $f$ with rate $\lambda$. In addition, note that
\begin{equation}\label{eq:diffpolj}
|L_\ell(x) - L_\ell(y)|
\leq |x - y|\cdot \frac{\max_j g_{\ell,j}}{m^{-\ell}} 
\leq \frac{|x - y|}{m^{-\ell}} K \lambda^\ell,
\end{equation}
implying
\begin{equation}\label{eq:difflimfunc}
|f(x) - f(y)| \leq |f(x) - L_\ell(x)| + |L_\ell(x) - L_\ell(y)| + |L_\ell(y) - f(y)|.
\end{equation}
It suffices to verify the H\"older condition locally. Let $x, y$ be such that $m^{-\ell} \leq |x - y| \leq m^{-\ell + 1}$, so that
\[ \lambda^\ell = m^{\ell \log_m(\lambda)} \leq |x - y|^{-\log_m(\lambda)},\qquad \frac{|x - y|}{m^{-\ell}} \leq m.\]
Combining \eqref{eq:difflimitpol}--\eqref{eq:difflimfunc} gives
\[ |f(x) - f(y)| \leq (2K_2 + m K)|x - y|^{-\log_m(\lambda)}, \]
implying that $f\in C^{-\log_m(\lambda)}$ whenever $1/m < \lambda < 1$.
\end{proof}

\subsection{Growth rate of the differences of the data}\label{sec:reduction}

How can we use \eqref{eq:growth} in the case that it holds with
$\lambda \ge 1$? Then we do not know whether $f \in C^r$, but if
$r \ge 1$ we
can use the `reduction procedure' of Daubechies,
Guskov, and Sweldens \cite{Daubechies.Guskov.Sweldens99} to
obtain information about lower order derivatives.
Although the procedure was shown to work for
binary interpolatory schemes in \cite{Daubechies.Guskov.Sweldens99},
it also applies to the more general scheme \eqref{eq:scheme}.

\begin{lemma}\label{lem:reductionprocedure}
Suppose \eqref{eq:polynomialgeneration2} holds for some $r \ge 1$.
If \eqref{eq:growth} holds for some $\lambda$, then there are constants $D_1, D_2, D_3$ such that
\[
|g_{\ell,j}^{[r-1]}| \leq
\left\{
\begin{array}{rl}
\displaystyle D_1    \left(\frac{\lambda}{m}\right)^\ell & \text{ if }\lambda > 1,\\
\displaystyle (D_2 + D_3\ell) \left(\frac{1}{m}\right)^\ell & \text{ if }\lambda = 1.
\end{array}
\right.
\]
\end{lemma}

\begin{proof}
By the definition of $a^{[r]}(z)$, one has $a^{[r]}(1) = a(1) = m$. Moreover, 
by the divisibility assumption \eqref{eq:polynomialgeneration2}, $a^{[r]}$ is divisible by
$(1 - z^m)/(1 - z)$ so that $a^{[r]}(\zeta_m^k) = 0$ for $k = 1,\ldots, m-1$.
It follows that
\[ \sum_k a_{mk}^{[r]} = \sum_k a_{mk+1}^{[r]} = \cdots = \sum_k a_{mk+m-1}^{[r]} = 1, \]
which, together with \eqref{eq:derivedschemea}, implies that there is a constant $K_1$ such that
\[ |f_{\ell+1,mj+s}^{[r]} - f_{\ell,j}^{[r]}| \le
   K_1 \max_j |g_{\ell,j}^{[r]}|, \qquad s=0,1,\ldots,m-1. \]
So, for any level $\ell \ge 1$, if we represent any $j \in \ZZ$ in $m$-ary form
as $j = j_\ell$, where
\[ j_{\ell'} = m j_{\ell'-1} + s_\ell, \qquad {\ell'}=\ell,\ell-1,\ldots,1, \]
for some $j_0 \in \ZZ$ and $s_1,\ldots,s_\ell \in \{0,1,\ldots,m-1\}$, then
\[ |f_{\ell,j}^{[r]} - f_{0,j_0}^{[r]}|
  \le \sum_{\ell'=1}^\ell | f_{\ell',j_{\ell'}}^{[r]} - f_{\ell'-1,j_{\ell'-1}}^{[r]} |
 \le K_1 K (1 + \lambda + \cdots + \lambda^{\ell-1}). \]
Hence,
\[ |f_{\ell,j}^{[r]}| \le K_2 + K_1 K (1 + \lambda + \cdots + \lambda^{\ell-1}) \]
for some constant $K_2$, and since
\[ g_{\ell,j}^{[r-1]} = m^{-\ell} r f_{\ell,j}^{[r]}, \]
this gives the result in the two cases $\lambda > 1$ and $\lambda = 1$.
\end{proof}

By applying this procedure recursively, it follows that
if \eqref{eq:growth} holds for any $\lambda$ with
$1/m < \lambda < m^r$, then
$f \in C^{r-\log_m(\lambda)}$ if $\log_m(\lambda)$ is not an integer,
and
$f \in C^{r-\log_m(\lambda) - \epsilon}$ for any small $\epsilon > 0$
if $\log_m(\lambda)$ is an integer.

\subsection{Growth rate of the iterated scheme for the differences}
\label{sec:rm}

If $p=0$ in \eqref{eq:bform}, then $b_0 = 1$ by \eqref{eq:convergence2} and \eqref{eq:polynomialgeneration2}. In this case the scheme \eqref{eq:scheme} is the $m$-ary B-spline scheme of degree $r$. Since \eqref{eq:growth} holds with $\lambda = 1$, we conclude using Lemma \ref{lem:reductionprocedure} that the limit function belongs to $C^\beta$ for any $\beta < r$, which is well known.

Therefore we assume from now on that $p \ge 1$. With $r$ in \eqref{eq:polynomialgeneration2} fixed, write $g_{\ell,j} = g_{\ell,j}^{[r]}$ and
$g_\ell(z) = \sum_j g_{\ell,j} z^j$. Then
\begin{equation}\label{eq:derivedschemeL}
 g_{\ell+1}(z) = b(z) g_\ell(z^m),
\end{equation}
with $b(z)$ as in \eqref{eq:polynomialgeneration2},
or equivalently,
\begin{equation}\label{eq:derivedscheme}
 g_{\ell+1,j} = \sum_k b_{j-mk} g_{\ell,k}.
\end{equation}
In the following lemma we rephrase the bound \eqref{eq:growth} for the data $g_{\ell,j}$ as a bound for their scheme.

\begin{lemma}\label{lem:growth2}
The bound \eqref{eq:growth} holds, for some constant $K$, if there is some constant $K'$ such that
\begin{equation}\label{eq:growth2}
 \max_j |b_{\ell,j}| \le K' \lambda^\ell.
\end{equation}
\end{lemma}

\begin{proof}
Iterating \eqref{eq:derivedschemeL} gives
\begin{equation}\label{eq:Gj}
 g_\ell(z) = b_\ell(z) g_0(z^{m^\ell}),
\end{equation}
where
\begin{equation}\label{eq:bj}
 b_\ell(z) := b(z) b(z^m) \cdots b(z^{m^{\ell-1}}).
\end{equation}
But then
\begin{equation}\label{eq:bj1}
 b_{\ell+1}(z) = b(z) b_\ell(z^m),
\end{equation}
and so $b_\ell(z)$ is the Laurent polynomial of the data $b_{\ell,j}$, where $b_{0,j} = \delta_{j,0}$ and
\begin{equation}\label{eq:schemeb}
 b_{\ell+1,j} = \sum_k b_{j-mk} b_{\ell,k}.
\end{equation}
In particular, $b_{1,j} = b_j$. Since \eqref{eq:Gj} can be written as
\begin{equation}\label{eq:gb}
 g_{\ell,j} = \sum_k b_{\ell,j-m^\ell k} g_{0,k},
\end{equation}
it follows that
\[ |g_{\ell,j}| \le \max_j |b_{\ell,j}| \sum_k |g_{0,k}|, \]
and so \eqref{eq:growth} holds if \eqref{eq:growth2} holds for some constant $K'$.
\end{proof}

The following lemma provides the reason why the bound \eqref{eq:growth2} is easier to verify than \eqref{eq:growth}, in the case of a nonnegative Fourier transform. For a direct proof see the report \cite{Floater.Muntingh13} or \cite{Rioul92}. It is also a direct consequence of Herglotz' theorem, which states that the condition of the lemma is equivalent to $\bfb$ being a positive definite sequence; see \cite{Charina14}.

\begin{lemma}\label{lem:rioul}
If $\bfb$ as in \eqref{eq:bform} has Fourier transform $B(\xi) \ge 0$ for all $\xi$, then
\[ \max_j |b_{\ell,j}| = b_{\ell,0} \qquad \text{for all } \ell \ge 0. \]
\end{lemma}

\subsection{Growth rate as a spectral radius}\label{sec:spectral}

For an odd symmetric mask $\bfb$ with nonnegative Fourier transform, it follows that \eqref{eq:growth} holds if $b_{\ell,0} \le K \lambda^\ell$ for large enough $\ell$. One way to determine such $\lambda$ is using a subvector of $[b_{\ell,j}]_j$ that includes the central coefficients $b_{\ell,0}$ and is `self-generating' in the following sense.

\begin{lemma}\label{lem:selfgenerating}
For $\ell\geq 0$, the finite submatrix and subvectors
\begin{equation}\label{eq:Mandb}
\bfM := [b_{j-mk}]_{j,k=-\lfloor\frac{p-1}{m-1}\rfloor,\ldots,\lfloor\frac{p-1}{m-1}\rfloor},\quad 
\bfb_\ell := \left[b_{\ell,-\lfloor\frac{p-1}{m-1}\rfloor},\ldots,b_{\ell,\lfloor\frac{p-1}{m-1}\rfloor}\right]^T,
\end{equation}
satisfy $\bfb_{\ell+1} = \bfM \bfb_\ell$.
\end{lemma}
\begin{proof}
If $k\geq \lfloor\frac{p-1}{m-1}\rfloor + 1$ in \eqref{eq:schemeb} and the corresponding
coefficient $b_{j-mk}\neq 0$, then $j - mk \geq -p$ implying that
\[ j\geq -p + mk\geq - p + m\left\lfloor\frac{p-1}{m-1}\right\rfloor + m \geq \left\lfloor\frac{p-1}{m-1}\right\rfloor + 1.\]
So any such $b_{\ell,k}$ will not contribute 
to the linear combination for $b_{\ell+1,j}$ with $j\le \lfloor\frac{p-1}{m-1} \rfloor$.
Similarly, if $k\leq -\lfloor\frac{p-1}{m-1}\rfloor - 1$
in \eqref{eq:schemeb} and the corresponding
coefficient $b_{j-mk}\neq 0$, then $j - mk \leq p$ implying that
\[ j \leq p + mk\leq p - m\left\lfloor \frac{p-1}{m-1} \right\rfloor - m
\le - \left\lfloor \frac{p-1}{m-1} \right\rfloor - 1.\]
So any such $b_{\ell,k}$ will not contribute 
to the linear combination for $b_{\ell+1,j}$ with $j\ge - \left\lfloor \frac{p-1}{m-1} \right\rfloor$. By \eqref{eq:schemeb}, it follows that $\bfb_{\ell+1} = \bfM \bfb_\ell$ for $\ell\geq 0$.
\end{proof}

\begin{theorem}\label{thm:MainTheoremLower}
Let $\rho$ be the spectral radius of $\bfM$. If $B(\xi) \ge 0$ for all $\xi$, then
\begin{equation}\label{eq:bj0lim}
 \lim_{\ell \to \infty} b_{\ell,0}^{1/\ell} = \rho,
\end{equation}
If $\rho > 1/m$, a lower bound for the regularity of the scheme \eqref{eq:scheme} is $r - \log_m(\rho)$.
\end{theorem}

\begin{proof}
Using Lemma \ref{lem:rioul} and $\bfb_\ell = \bfM^\ell \bfb_0$ by Lemma \ref{lem:selfgenerating},
\[ b_{\ell,0} = \Vert \bfb_\ell \Vert_\infty
   \le \Vert \bfM^\ell \Vert_\infty \Vert \bfb_0 \Vert_\infty
   = \Vert \bfM^\ell \Vert_\infty. \]
On the other hand, by \eqref{eq:gb} the matrix $\bfM^\ell$ takes its entries from $[b_{\ell,j}]_j$, so that its maximum absolute row sum $\Vert \bfM^\ell \Vert_\infty$ satisfies
\[ \Vert \bfM^\ell \Vert_\infty \le \left(2\left\lfloor \frac{p-1}{m-1} \right\rfloor + 1\right) \max_j |b_{\ell,j}|
  =  \left(2\left\lfloor \frac{p-1}{m-1} \right\rfloor + 1\right) b_{\ell,0}. \]
Taking $\ell$-th roots and the limit $\ell\to \infty$ one obtains \eqref{eq:bj0lim}. It follows from \eqref{eq:bj0lim} and Lemma \ref{lem:growth2} that \eqref{eq:growth} holds with $K = 1$ for any $\lambda > \rho$, and this proves the lower bound on the regularity of the scheme.
\end{proof}

\subsection{A smaller matrix}\label{sec:folded}

Due to the assumption that $\bfb$ is odd symmetric, the limit \eqref{eq:bj0lim} can also be computed as the spectral radius of a matrix roughly half the size of $\bfM$, using a `folding procedure' \cite{Floater.Muntingh13, Rioul92}. Since $b_{\ell,-j} = b_{\ell,j}$ for all $j$, the vector of coefficients
\[ \bfb_\ell := \left[b_{\ell,0},b_{\ell,1},\ldots,b_{\ell,\lfloor\frac{p-1}{m-1}\rfloor}\right]^T, \]
also includes $b_{\ell,0}$ and is self-generating as well.
Indeed, from \eqref{eq:schemeb},
\[ b_{\ell+1,j} = b_j b_{\ell,0} + \sum_{k \ge 1} 
         (b_{j-mk} + b_{j+mk}) b_{\ell,k}, \]
and, using that $b_{-j} = b_j$, one obtains
\[ b_{\ell+1,j} = b_j b_{\ell,0} + \sum_{k \ge 1} 
         (b_{|j-mk|} + b_{j+mk}) b_{\ell,k}. \]
It follows that $\bfb_{\ell+1} = \bfM \bfb_\ell$, where $\bfM$ is the matrix of dimension $\lfloor \frac{p-1}{m-1}\rfloor + 1$,
\begin{equation}\label{eq:Msmall}
\bfM = [m_{j,k}]_{j,k=0,\ldots,\lfloor \frac{p-1}{m-1}\rfloor},
\qquad
m_{j,k} = \left\{\begin{array}{ll}
b_j & k = 0,\\
b_{|j-mk|} + b_{j+mk} & k \ge 1.
\end{array} \right.
\end{equation}

\subsection{Optimality}\label{sec:optimality}

In this section we show that under a slightly stricter condition, the lower bound on the regularity of Theorem~\ref{thm:MainTheoremLower} is optimal. For related results in the binary and ternary case, see \cite{Rioul92} and \cite[\S 3.4]{Hassan05}.

\begin{theorem}\label{thm:MainTheoremStrict}
If $B(\xi) > 0$ for all~$\xi$, the lower bound $r - \log_m(\rho)$ of Theorem~\ref{thm:MainTheoremLower} is optimal.
\end{theorem}

To prove this we first establish a lemma that shows that the bound is optimal whenever the cardinal function $\phi$ of the scheme \eqref{eq:scheme} has $\ell^\infty$-stable integer translates. The main point in proving this lemma is that the stability allows us to bound divided differences of the scheme by corresponding divided differences of the limit function.

Following Jia and Micchelli \cite{Jia.Micchelli90}, we say that $\phi$
\emph{has $\ell^\infty$-stable integer translates} if there is some
constant $K_\infty > 0$ such that
for any sequence $\bfc = [c_j]_j$ in $\ell^\infty(\ZZ)$,
\begin{equation}\label{eq:stability}
   \left\| \sum_j c_j \phi(\cdot - j) \right\|_{L^\infty(\RR)} 
    \ge K_\infty \Vert\bfc\Vert_{\ell^\infty(\ZZ)}.
\end{equation}

\begin{lemma}\label{lem:stable}
Suppose $\phi$ has $\ell^\infty$-stable integer translates and $f\in C^{q + \alpha}$ for some integer $q\geq 0$ and $0 < \alpha < 1$. Then for any integer $r \ge q$, there is a constant $K$ such that
\begin{equation}\label{eq:growthopt}
 |g_{\ell,j}^{[r]}| \le K m^{\ell(r-q-\alpha)}.
\end{equation}
\end{lemma}
\begin{proof}
The limit function for general initial data $f_{0,k}$ can be expressed as the
linear combination
\[ f(x) = \sum_k f_{0,k} \phi(x-k). \]
As is well known \cite{Han.Jia98}, $\phi$ satisfies the refinement equation
\begin{equation}\label{eq:multiscale}
 \phi(x) = \sum_j a_j \phi(mx - j),
\end{equation}
and therefore, for any $\ell \ge 0$,
\begin{equation}\label{eq:useful}
 f(x) = \sum_k f_{\ell,k} \phi(m^\ell x - k).
\end{equation}
We can use this equation to relate any divided difference
of $f$ of the form
\[ \tilde f_{\ell, y}^{[q]} := [m^{-\ell}(y-q),m^{-\ell}(y-q+1),\ldots,m^{-\ell}y]f,\qquad y\in \RR, \]
to the divided differences of the scheme.
Putting $x = m^{-\ell}(y-j)$ in \eqref{eq:useful},
\[ f\big(m^{-\ell}(y-j)\big) = \sum_k f_{\ell,k-j} \phi(y-k), \]
and, using the cases $j=0,1,\ldots,q$, and
the linearity of divided differences,
\[ \tilde f_{\ell,y}^{[q]} = \sum_k f_{\ell,k}^{[q]} \phi(y-k). \]
Similarly, if
\[ \tilde g_{\ell,y}^{[q]} := \tilde f_{\ell,y}^{[q]} - \tilde f_{\ell,y-1}^{[q]},\]
then
\[ \tilde g_{\ell,y}^{[q]} = \sum_k g_{\ell,k}^{[q]} \phi(y-k). \]

Using that $f$ has compact support, if $f$ has regularity $q + \alpha$, there is a constant $K'>0$ such that for any $\xi_0,\xi_1 \in \RR$,
\[ |f^{(q)}(\xi_1) - f^{(q)}(\xi_0)| \le K' |\xi_1 - \xi_0|^\alpha, \]
and, by the mean value theorem for divided differences, for each $\ell$ and $y$,
\[ |\tilde g_{\ell,y}^{[q]}| = |f^{(q)}(\xi_1) - f^{(q)}(\xi_0)| / q!,\]
for $\xi_0,\xi_1 \in \big(m^{-\ell}(y-q-1),m^{-\ell}y\big)$. Therefore, for any $y$,
\[ |\tilde g_{\ell,y}^{[q]}| \le K'' m^{-\ell\alpha}, \]
where $K'' = K' (q+1)^\alpha / q!$. Therefore,
\[ \left\| \sum_\ell g_{\ell,k}^{[q]} 
        \phi(\cdot - k) \right\|_{L^\infty(\RR)} 
    \le K'' m^{-\ell\alpha},\]
and by \eqref{eq:stability} it follows that for any $k \in \ZZ$,
\[ |g_{\ell,k}^{[q]}| \le K^{-1}_\infty K'' m^{-\ell\alpha}. \]
Finally, by applying the divided difference definitions
\eqref{eq:recurse} recursively, $r-q$ times, we obtain \eqref{eq:growthopt}.
\end{proof}

\begin{lemma}\label{lem:stable2}
If $\phi$ has $\ell^\infty$-stable integer translates, then the lower bound $r - \log_m(\rho)$ of Theorem~\ref{thm:MainTheoremLower} is optimal.
\end{lemma}

\begin{proof}
Let $f$ be the limit of the scheme with any initial data for which $g_{0,j}^{[r]} = \delta_{j,0}$, $-p+1 \le j \le p-1$, and with only a finite number of initial data $f_{0,j}$ non-zero. Then $f$ has compact support.
Suppose that $f \in C^{r-\log_m(\rho) + \varepsilon}$ for some small $\varepsilon > 0$ and write the exponent as
\[ r-\log_m(\rho) + \varepsilon = q + \alpha, \qquad q \in \NN_0, \qquad 0 < \alpha < 1.\]
If $\rho > 1/m$, we have $r \ge q$, and so Lemma~\ref{lem:stable} can be applied, implying
\[ |g_{\ell,j}^{[r]}| \le K m^{\ell(\log_m(\rho) - \varepsilon)}
      = K \rho^\ell m^{-\ell\varepsilon}. \]
Hence,
\[ \limsup_{\ell\to \infty} \left|g_{\ell,0}^{[r]}\right|^{1/\ell}  \le \rho m^{-\varepsilon}. \]
By choice of the $g_{0,j}^{[r]}$, however, $g_{\ell,0}^{[r]} = b_{\ell,0}$, which contradicts \eqref{eq:bj0lim}.
\end{proof}

Using this lemma we can now prove Theorem~\ref{thm:MainTheoremStrict} by comparing the cardinal
function $\phi$ with B-splines, which are known to be stable. A similar idea was used by Dong and Shen \cite[Lemma 2.2]{Dong.Shen06} to show that binary pseudo-splines are stable.

\begin{proof}[Proof of Theorem \ref{thm:MainTheoremStrict}]
By Lemma~\ref{lem:stable2}, it is sufficient to show that $\phi$ has $\ell^\infty$-stable integer translates if $B(\xi) > 0$ for all $\xi$. We apply some results by Jia and Micchelli \cite{Jia.Micchelli90}. Consider the (continuous) Fourier transform of $\phi$, defined as
\[ \wphi(\xi) := \int_{\RR} \phi(x) \rme^{-i\xi x} \, \rmd x, \qquad \xi \in \RR. \]
Since the scheme \eqref{eq:scheme} reproduces constants,
\[ \sum_k \phi(x-k) = 1, \qquad x \in \RR. \]
As a $1$-periodic function, it has a Fourier series expansion
\[ \sum_k \phi(x-k) = \sum_{n\in \ZZ} c_n \rme^{2\pi i n x}, \]
with Fourier coefficients
\[
\delta_{n,0}
= c_n
= \int_0^1 \sum_k \phi(x-k) \rme^{-2\pi i n x} \rmd x
= \int_\RR \phi(x) \rme^{-2\pi i n x} \rmd x = \wphi(2\pi n).
\]
In particular $\wphi(0) = 1$. Together with the Fourier transform of \eqref{eq:multiscale},
\[ \wphi(\xi) =  m^{-1} A(\xi/m) \wphi(\xi/m), \]
if follows that
\[ \wphi(\xi) = \prod_{\ell=1}^\infty \big(m^{-1} A(\xi/m^\ell)\big). \]

By \cite[Theorem 3.5]{Jia.Micchelli90}, $\phi$ has $\ell^\infty$-stable integer translates precisely when 
\begin{equation}\label{eq:sup}
 \sup_{k \in \ZZ} \left|\wphi(\xi + 2 \pi k)\right| > 0,
 \qquad \hbox{for all } \xi \in \RR.
\end{equation}
Consider again the case that the scheme admits a factorization \eqref{eq:polynomialgeneration2}. Then
\[ A(\xi) = m \rme^{-(m-1)(r+1)i\xi/2} \left(\frac{\sin(m\xi/2)}{m\sin(\xi/2)}\right)^{r+1}B(\xi), \]
where, since $A(0) = m$ under the assumption of convergence, $B(0) = 1$.
For the B-spline scheme of degree $r$ we have $b(z) = 1$, in which case we can write its symbol as $a_r(z) = (1+z+\cdots + z^{m-1})^{r+1} / m^r$.
The cardinal function $\phi_r$ is the B-spline of degree $r$ centered at 0, and we have, after shifting,
\begin{align*}
\rme^{(r+1)i\xi/2} \cdot \wphi_r(\xi)
& = \prod_{\ell=1}^\infty \left(\frac{\sin(m^{-\ell+1}\xi/2 )}{m\sin(m^{-\ell} \xi/2 )} \right)^{r+1} \\
& = \left(\frac{\sin(\xi/2)}{\xi/2}\right)^{r+1} \lim_{\ell\to \infty} \left(\frac{m^{-\ell} \xi/2}{\sin(m^{-\ell} \xi/2 )}\right)^{r+1}\\
& = \left(\frac{\sin(\xi/2)}{\xi/2}\right)^{r+1}
  =: \sinc^{r+1}(\xi/2).
\end{align*}
It then follows that
\[ \wphi(\xi) = \wphi_r(\xi)
      \prod_{\ell=1}^\infty B\big(\xi/m^\ell\big). \]
Since the condition \eqref{eq:sup} holds for the B-spline $\phi_r$, we deduce that $\phi$ has $\ell^\infty$-stable integer translates if $B(\xi) > 0$ for all $\xi$.
\end{proof}

\begin{bibdiv}
\begin{biblist}

\bib{Cavaretta.Dahmen.Micchelli91}{article}{
   author={Cavaretta, Alfred S.},
   author={Dahmen, Wolfgang},
   author={Micchelli, Charles A.},
   title={Stationary subdivision},
   journal={Mem. Amer. Math. Soc.},
   volume={93},
   date={1991},
   number={453},
   pages={vi+186},
}

\bib{Charina14}{article}{
   author={Charina, Maria},
   title={Finiteness conjecture and subdivision},
   journal={Applied and Computational Harmonic Analysis},
   volume={36},
   number={3},
   pages={522--526},
   year={2014},
}

\bib{Conti.Gemignani.Romani16}{article}{
   author={Conti, Costanza},
   author={Gemignani, Angelo},
   author={Romani, Lucia},
   title={Exponential pseudo-splines: Looking beyond exponential B-splines},
   journal={Journal of Mathematical Analysis and Applications},
   date={2016},
   volume={439},
   number={1},
   pages={32--56},
}

\bib{Conti.Hormann11}{article}{
   author={Conti, Costanza},
   author={Hormann, Kai},
   title={Polynomial reproduction for univariate subdivision schemes of any
   arity},
   journal={J. Approx. Theory},
   volume={163},
   date={2011},
   number={4},
   pages={413--437},
   issn={0021-9045},
}

\bib{Daubechies.Guskov.Sweldens99}{article}{
   author={Daubechies, Ingrid},
   author={Guskov, Igor},
   author={Sweldens, Wim},
   title={Regularity of irregular subdivision},
   journal={Constr. Approx.},
   volume={15},
   date={1999},
   number={3},
   pages={381--426},
   issn={0176-4276},
}

\bib{Daubechies.Han.Ron.Shen03}{article}{
   author={Daubechies, Ingrid},
   author={Han, Bin},
   author={Ron, Amos},
   author={Shen, Zuowei},
   title={Framelets: MRA-based constructions of wavelet frames},
   journal={Appl. Comput. Harmon. Anal.},
   volume={14},
   date={2003},
   number={1},
   pages={1--46},
   issn={1063-5203},
}

\bib{Deng.Hormann14}{article}{
   author={Deng, Chongyang}
   author={Hormann, Kai}
   title={Pseudo-Spline Subdivision Surfaces}
   journal={Computer Graphics Forum},
   volume={33},
   number={5},
   issn={1467-8659},
   year={2014},
   url={http://dx.doi.org/10.1111/cgf.12448},
   pages={227--236},
}

\bib{Deslauriers.Dubuc89}{article}{
   author={Deslauriers, Gilles},
   author={Dubuc, Serge},
   title={Symmetric iterative interpolation processes},
   journal={Constr. Approx.},
   volume={5},
   date={1989},
   number={1},
   pages={49--68},
   issn={0176-4276},
}

\bib{Dyn.Levin.Gregory87}{article}{
   title={A 4-point interpolatory subdivision scheme for curve design},
   author={Dyn, Nira},
   author={Levin, David},
   author={Gregory, John A.},
   journal={Computer Aided Geometric Design}
   volume = {4},
   number = {4},
   pages = {257--268},
   year = {1987}
}

\bib{Dong.Dyn.Hormann10}{article}{
   author={Dong, Bin},
   author={Dyn, Nira},
   author={Hormann, Kai},
   title={Properties of dual pseudo-splines},
   journal={Appl. Comput. Harmon. Anal.},
   volume={29},
   date={2010},
   number={1},
   pages={104--110},
   issn={1063-5203},
}

\bib{Dong.Shen06}{article}{
   author={Dong, Bin},
   author={Shen, Zuowei},
   title={Linear independence of pseudo-splines},
   journal={Proc. Amer. Math. Soc.},
   volume={134},
   date={2006},
   number={9},
   pages={2685--2694},
   issn={0002-9939},
}

\bib{Dong.Shen07}{article}{
   author={Dong, Bin},
   author={Shen, Zuowei},
   title={Pseudo-splines, wavelets and framelets},
   journal={Appl. Comput. Harmon. Anal.},
   volume={22},
   date={2007},
   number={1},
   pages={78--104},
   issn={1063-5203},
}

\bib{Dubuc86}{article}{
   author={Dubuc, Serge},
   title={Interpolation through an iterative scheme},
   journal={J. Math. Anal. Appl.},
   volume={114},
   date={1986},
   number={1},
   pages={185--204},
   issn={0022-247X},
}

\bib{Dyn.Floater.Hormann04}{article}{
   author={Dyn, Nira},
   author={Floater, Michael S.},
   author={Hormann, Kai},
   title={A $C^2$ four-point subdivision scheme with fourth order
   accuracy and its extensions},
   conference={
      title={Mathematical methods for curves and surfaces: Troms\o\ 2004},
   },
   book={
      series={Mod. Methods Math.},
      publisher={Nashboro Press, Brentwood, TN},
   },
   date={2005},
   pages={145--156},
}

\bib{Dyn.Hormann.Sabin.Shen08}{article}{
   author={Dyn, Nira},
   author={Hormann, Kai},
   author={Sabin, Malcolm A.},
   author={Shen, Zuowei},
   title={Polynomial reproduction by symmetric subdivision schemes},
   journal={J. Approx. Theory},
   volume={155},
   date={2008},
   number={1},
   pages={28--42},
   issn={0021-9045},
}

\bib{Dyn.Levin02}{article}{
   author={Dyn, Nira},
   author={Levin, David},
   title={Subdivision schemes in geometric modelling},
   journal={Acta Numer.},
   volume={11},
   date={2002},
   pages={73--144},
   issn={0962-4929},
}

\bib{Floater11}{article}{
   author={Floater, Michael S.},
   title={A piecewise polynomial approach to analyzing interpolatory
   subdivision},
   journal={J. Approx. Theory},
   volume={163},
   date={2011},
   number={11},
   pages={1547--1563},
   issn={0021-9045},
}

\bib{Floater.Muntingh13}{article}{
   author={Floater, Michael S.},
   author={Muntingh, Georg},
   title={Exact regularity of pseudo-splines},
   eprint={http://arxiv.org/pdf/1209.2692.pdf}
}

\bib{Floater.Siwek13}{article}{
   author={Floater, Michael S.},
   author={Siwek, Bart{\l}omiej P.},
   title={Analysis of Hermite subdivision using piecewise polynomials},
   journal={BIT Numerical Mathematics},
   year={2013},
   volume={53},
   number={2},
   pages={397--409},
   issn={1572-9125},
   url={http://dx.doi.org/10.1007/s10543-012-0418-9}
}

\bib{Han02}{article}{
   author={Han, Bin},
   title={Computing the smoothness exponent of a symmetric multivariate refinable function},
   journal = {SIAM J. Matrix Anal. Appl.},
   volume = {24},
   number = {3},
   month = {mar},
   year = {2002},
   issn = {0895-4798},
   pages = {693--714},
   url = {http://dx.doi.org/10.1137/S0895479801390868},
}

\bib{Han.Jia98}{article}{
   author={Han, Bin},
   author={Jia, Rong-Qing},
   title={Multivariate refinement equations and convergence of subdivision
   schemes},
   journal={SIAM J. Math. Anal.},
   volume={29},
   date={1998},
   number={5},
   pages={1177--1199 (electronic)},
   issn={0036-1410},
}

\bib{Hassan05}{thesis}{
   author = {Hassan, Mohamed F.},
   title = {Multiresolution in Geometric Modelling: Subdivision Mark Points and Ternary Subdivision},
   type = {phd},
   date = {2005},
   organization={University of Cambridge},
}

\bib{Hassan.Dodgson02}{article}{
   author = {Hassan, Mohamed F.},
   author = {Dodgson, Neil A.},
   title = {Ternary and three-point univariate subdivision schemes}
   conference={
      title={Curve und surface fitting, Saint-Malo 2002. Fifth international conference on curves and surfaces},
   },
   book={
      publisher={Nashboro Press, Brentwood, TN},
   },
   date={2003},
   pages={199--208},
}


\bib{Hechler.Moessner.Reif09}{article}{
   title = {C1-continuity of the generalized four-point scheme},
   author = {Hechler, Jochen},
   author = {M\"o\ss ner, Bernhard},
   author = {Reif, Ulrich},
   journal = {Linear Algebra and its Applications},
   volume = {430},
   number = {11},
   pages = {3019--3029},
   year = {2009},
}

\bib{Ivrissimtzis.Sabin.Dodgson04}{article}{
   author={Ivrissimtzis, Ioannis},
   author={Sabin, Malcolm A.},
   author={Dodgson, Neil A.},
   title={On the support of recursive subdivision},
   journal={ACM Transactions on Graphics},
   volume={23}, 
   number={4}, 
   pages={1043--1060},
   date={October 2004}
}

\bib{Jia.Micchelli90}{article}{
   author={Jia, Rong Qing},
   author={Micchelli, Charles A.},
   title={Using the refinement equations for the construction of
   pre-wavelets. II. Powers of two},
   conference={
      title={Curves and surfaces},
      address={Chamonix-Mont-Blanc},
      date={1990},
   },
   book={
      publisher={Academic Press, Boston, MA},
   },
   date={1991},
   pages={209--246},
}

\bib{Ko.Lee.Yoon07}{article}{
   author={Ko, Kwan Pyo},
   author={Lee, Byung-Gook},
   author={Yoon, Gang Joon},
   title={A ternary 4-point approximating subdivision scheme},
   journal={Appl. Math. Comput.},
   volume={190},
   date={2007},
   number={2},
   pages={1563--1573},
   issn={0096-3003},
}

\bib{Lian09}{article}{
   author={Lian, Jian-Ao},
   title={On $\alpha$-ary subdivision for curve design. III. $2m$-point and $(2m+1)$-point interpolatory schemes},
   journal={Applications and Applied Mathematics},
   volume={4},
   number={2},
   pages={434--444},
   date={2009}
}

\bib{Moeller15}{thesis}{
   author = {Claudia M\"oller},
   title={A New Strategy for Exact Determination of the Joint Spectral Radius},
   type={phd},
   date={2015},
   organization={Technische Universit\"at Darmstadt},
   eprint={http://tuprints.ulb.tu-darmstadt.de/4603/}
}

\bib{Moeller.Reif14}{article}{
   author = {Claudia M\"oller},
   author = {Ulrich Reif},
   title = {A tree-based approach to joint spectral radius determination},
   journal = {Linear Algebra and its Applications},
   volume = {463},
   pages = {154--170},
   date = {2014},
}

\bib{WebsiteGeorg}{article}{
   author={Muntingh, Georg},
   title={Personal Website},
   eprint={https://sites.google.com/site/georgmuntingh/academics/software}
}

\bib{Mustafa.Khan09}{article}{
   author={Mustafa, Ghulam},
   author={Khan, Faheem},
   title={A new 4-point $C^3$ quaternary approximating subdivision
   scheme},
   journal={Abstr. Appl. Anal.},
   date={2009},
   pages={Art. ID 301967, 14},
   issn={1085-3375},
}

\bib{Rioul92}{article}{
   author={Rioul, Olivier},
   title={Simple regularity criteria for subdivision schemes},
   journal={SIAM J. Math. Anal.},
   volume={23},
   date={1992},
   number={6},
   pages={1544--1576},
   issn={0036-1410},
}

\bib{Zheng.Hu.Peng09}{article}{
   author={Zheng, Hongchan},
   author={Hu, Meigui},
   author={Peng, Guohua},
   title={Constructing $(2n-1)$-point ternary interpolatory subdivision schemes by using variation of constants},
   conference={
      title={International Conference on Computational Intelligence and Software Engineering (CiSE '09)},
   },
   date = {2009},
}
\end{biblist}
\end{bibdiv}
\end{document}